\documentclass[11pt]{article}

\usepackage{mathrsfs}
\usepackage{kbordermatrix}
\usepackage{enumerate}
\usepackage[section]{algorithm}
\usepackage{algorithmic}
\usepackage{multirow}
\usepackage{xcolor}

\usepackage{graphicx}
\usepackage{amsmath,amsfonts,amsthm,amsbsy,amssymb}
\usepackage{fixmath}
\usepackage{amssymb,latexsym}

\usepackage{hyperref}
\usepackage{cleveref}

\usepackage{slashbox}

\def\bp{\pmb{p}}

\def\bbC{\mathbb{C}}

\def\bbR{\mathbb{R}}

\def\scrH{\mathscr{H}}

\def\scrR{\mathscr{R}}

\def\cR{{\cal R}}

\def\pjbd{{\sc pjbd}}
\def\jbd{{\sc jbd}}
\def\pjd{{\sc pjd}}
\def\jd{{\sc jd}}

\newcommand\STM[2]{{\rm St}(#1,#2)}

\def\wtd{\widetilde}
\def\what{\widehat}

\usepackage{accents}

\DeclareMathOperator*{\argmax}{argmax}
\DeclareMathOperator{\BDiag}{BDiag}

\DeclareMathOperator{\diag}{diag}

\DeclareMathOperator{\grad}{grad}
\DeclareMathOperator{\init}{init}

\DeclareMathOperator*{\opt}{opt}
\DeclareMathOperator{\rank}{rank}
\DeclareMathOperator{\sym}{sym}
\DeclareMathOperator{\tot}{tot}
\DeclareMathOperator{\tr}{tr}

\DeclareMathOperator{\F}{F}
\DeclareMathOperator{\HH}{H}
\DeclareMathOperator{\T}{T}

\DeclareMathOperator{\KKT}{KKT}

\def\scrR{\mathscr{R}}

\def\Red{\textcolor{red}}

\newtheorem{theorem}{Theorem}[section]
\newtheorem{lemma}{Lemma}[section]

\theoremstyle{definition}

\newtheorem{remark}{Remark}[section]

\allowdisplaybreaks

\numberwithin{equation}{section}
\numberwithin{figure}{section}
\numberwithin{table}{section}

\def\sss{\scriptstyle}

\title{
An NPDo Approach for Principal Joint  Block Diagonalization
}

\author{Ren-Cang Li%
\thanks{Department of Mathematics, University of Texas at Arlington, Arlington, TX 76019-0408, USA.
        Supported in part by NSF DMS-2407692.
        Email: {\tt rcli@uta.edu}.}
\and
Ding Lu%
\thanks{Department of Mathematics, University of Kentucky, Lexington, KY 40506, USA. Supported in part by NSF DMS-2110731. Email: {\tt Ding.Lu@uky.edu}.}
\and
Li Wang%
\thanks{Department of Mathematics, University of Texas at Arlington, Arlington, TX 76019-0408, USA.
        Supported in part by NSF DMS-2407692.
        Email: {\tt Li.Wang@uta.edu}.}
\and
Lei-Hong Zhang%
\thanks{School of Mathematical Sciences, Soochow University, Suzhou 215006, Jiangsu, China.
 Supported in part by the National Natural Science Foundation of China (NSFC-12471356, NSFC-12371380), Jiangsu Shuangchuang Project (JSSCTD202209),  Academic Degree and Postgraduate Education Reform Project of Jiangsu Province, and China Association of Higher Education under grant 23SX0403.
        Email: {\tt longzlh@suda.edu.cn}.}
}

\date{
August 20, 2025
}

\begin{document}

\maketitle

\centerline{\em Dedicated to Prof. \AA ke Bj{\"o}rck on the occasion of his 90th birthday}

\begin{abstract}
Matrix joint
block-diagonalization (\jbd)
frequently arises from diverse applications such as
independent component analysis,
blind source separation,
and common principal component analysis (CPCA),
among others.
Particularly, CPCA aims at joint diagonalization, i.e.,
each block size being $1$-by-$1$.
This paper is concerned with {\em  principal joint block-diagonalization\/} (\pjbd),
which aim to achieve two goals:
1)~partial joint block-diagonalization, and
2)~identification of dominant common block-diagonal parts for all involved matrices.
This is in contrast to most existing methods,
especially the popular ones based on Givens rotation,
which focus on full joint diagonalization and quickly become
impractical for matrices of even moderate size ($300$-by-$300$ or larger).
An  NPDo approach is proposed and it is
built on a {\em nonlinear polar decomposition with orthogonal polar factor dependency}
that characterizes the solutions of the optimization problem designed to achieve \pjbd,
and it is shown the associated SCF iteration is globally convergent to a stationary point
while the objective function increases monotonically during the iterative process.
Numerical experiments are presented to illustrate the effectiveness of
the NPDo approach and its superiority to Givens rotation-based methods.


\end{abstract}

\clearpage
\tableofcontents

\clearpage
\section{Introduction}\label{sec:intro}
%

Consider the problem of {\em joint block-diagonalization} (\jbd):
Given $N$ Hermitian matrices $A_{\ell}\in\bbC^{n\times n}$,
for $\ell=1,\dots,N$,
the goal is to find some unitary $P\in\bbC^{n\times n}$ such that
$P^{\HH}A_{\ell}P$,
for $\ell = 1,\dots,N$,
are block-diagonal with the same given block-diagonal structure.
Such a matrix $P$ is called a {\em block-diagonalizer}.
A notable special case is the {\em joint diagonalization} (\jd), where each $P^{\HH}A_{\ell}P$
is diagonal and $P$ is accordingly referred to as a {\em diagonalizer}.

In general, for $N\ge 2$, an exact (block-)diagonalizer does not exist
unless certain strong conditions are met.
For example, $\{A_{\ell}\}_{\ell = 1}^N$  are jointly diagonalizable if and only if
they pairwise commute,
i.e., $A_{\ell}A_j=A_jA_{\ell}$ for any $\ell\ne j$ (see, e.g.,~\cite{hojo:2013}).
In practice, however, we typically seek a unitary $P$ such that all
matrices $P^{\HH}A_{\ell}P$  are as close as possible,
under a suitable measure, to a desired block-diagonal form.
Thus,  the \jbd\ problem is inherently an
{\em approximate joint block-diagonalization} problem. For simplicity, we will
not distinguish between the ``exact'' and ``approximate'' \jbd\ in this paper
and treat all such diagonalization problems as approximate ones by default.

\jbd\ is a classical problem in numerical linear algebra
and arises from various important applications, including
independent component analysis, also known as {\em multidimensional independent component analysis\/} (MICA),
and blind source separation \cite{beac:1997,card:1998,card:1999,como:1994,dldv:2000,mosc:1994,phca:2001,thei:2005,thei:2006,wupr:1999,yere:2002,zimn:2000},
common principal component analysis (CPCA) \cite{aifl:1988,bopr:2010,brmc2014,fehv:2003,flur:1988,gu:2016,li:2019,meho:2003,pepl:2014,wida:2018},
semidefinite programming \cite{badp:2009,dkps:2007,dkso:2010,gapa:2004}, and
kernel-based  nonlinear blind source separation \cite{hazk:2003}.
In those applications, approximate block-diagonalizers 
play major roles.

The most widely used approaches for solving \jd\
are Jacobi-like algorithms based on (complex) planar rotations.
Originally developed for the symmetric eigenvalue problem (SEP) \cite{parl:1998},
Jacobi’s method became popular due to its simplicity, embarrassing
parallelism, and ability to achieve high accuracy~\cite{demm:1997,govl:2013,parl:1998}.
The basic idea is to apply a sequence of planar rotations to increase the  diagonal part,
or equivalently reduce the  off-diagonal part,
of a Hermitian matrix, so as to drive the matrix to the diagonal form.
This approach naturally extends to \jd\ \cite{bubm:1993,caso:1996},
where the goal is to increase the total diagonal part, or equivalently
reduce the total off-diagonal part, of all $N$ matrices.
However, unlike the single matrix case, it is not that straightforward
to compute the optimal rotations.
Cardoso and Souloumiac \cite{caso:1996} discovered an elegant method for
addressing this issue.  Their algorithm for \jd\ proceeds in sweeps,
each consisting of $n(n-1)/2$ planar rotations that cover all
the $n(n-1)/2$ index pairs $(i,j)$ for the off-diagonal positions.  The
method  stops when further improvement becomes negligible.
In a similar spirit, for positive definite and real matrices $A_{\ell}$,
Flury and Gautschi \cite{flga:1986} proposed an inner-outer iterative
scheme, called the FG-algorithm, that
iteratively solves the first-order optimality conditions for optimizing
a suitable objective function that measures simultaneous deviation from diagonality.
These optimality conditions take the form of $n(n-1)/2$ scalar equations \cite{flur:1984},
and their algorithm proceeds in sweeps,
each consisting of $n(n-1)/2$ planar rotations as well, with each rotation
targeting one of the first-order equations.
The FG-algorithm stops when all the first-order conditions are satisfied approximately.


We note in passing that, in broader settings, the  \jbd\ problem may also use
a merely nonsingular (not necessarily unitary) diagonalizer $P$.
Finding a block-diagonalizer $P$ that is only assumed nonsingular
is numerically more delicate, as numerical instability can arise unexpectedly.
As a result, extra care must be taken in the design and implementation of algorithms.
Several methods have been proposed to address this more general setting; see, e.g.,
\cite{abcg:2019,cali:2017,cali:2019,chkm:2014,dela:2009,pham:2001,tipc:2017,ziln:2004}
and references therein, but critical issues remain.


Most existing methods are about  full \jbd\ (or full \jd\ for that matter).
The main goal of this paper is to introduce the notion of {\em principal \jbd} (\pjbd) and
the {\em nonlinear polar decomposition with orthogonal polar factor dependency}
(NPDo) approach for its efficient computation.
 \pjbd\ is a type of {\em partial} {\jbd}, where
a subspace of a lower dimension $k<n$ is identified
to capture the most significant joint block-diagonal components across
all $\{A_{\ell}\}_{\ell=1}^N$,
and the matrices are jointly block-diagonalized over this subspace.
This is particularly useful when $n$ is large, for which  {\em full \jbd}
is prohibitively expensive to compute and unnecessary,
and it is sufficient to go for the dominant components,
similar in spirit to the principal component analysis (PCA) commonly used in
statistics and data science.
The general mathematical and numerical
foundation of the NPDo approach was recently developed in \cite{li:2024},
which also guarantees the global convergence of the underlying iterative process.
We note that both Jacobi's method in \cite{bubm:1993,caso:1996} and
the FG-algorithm \cite{flga:1986} are for full \jd\
and not directly applicable in settings where only a partial
\jd\ is of interest, although they do induce \pjd\ as we will explain
later in subsection~\ref{ssec:egs-pjd}.


The rest of this paper is organized as follows.
In \cref{sec:NPDo}, we study the numerical solution of a model optimization
problem on  Stiefel manifold via the NPDo approach,
which will serve as the computational engine for our main objective:
 efficient computation of \pjbd\ to be detailed in \cref{sec:NPDo4PJBD}.
Extensive numerical demonstrations are presented in \cref{sec:egs}.
Final conclusions are drawn in \cref{sec:concl}.

{\bf Notation.}
We follow conventional notation in matrix analysis.
The set $\bbR^{m\times n}$  denotes  $m\times n$ real matrices,
with $\bbR^n=\bbR^{n\times 1}$, and $\bbR=\bbR^1$.
Similarly, $\bbC^{m\times n}$,  $\bbC^n$, and $\bbC$ denote the
corresponding sets in the complex number field.
The identity matrix of size $n$ is denoted by $I_n$,
or simply $I$ if its size is clear from the context.
For any matrix or vector $B$,
$B^{\T}$ and $B^{\HH}$ stand for the transpose and  complex conjugate
transpose, respectively.
The (complex) Stiefel manifold $\STM{k}{n}$ consists
of all $P\in\bbC^{n\times k}$ with orthonormal columns:
\begin{equation}\label{eq:stm} 
\STM{k}{n}=\{P\in\bbC^{n\times
k}\,:\,P^{\HH}P=I_k\}\subset\bbC^{n\times k}.
\end{equation}
A square matrix $A\succ 0$ (resp., $\succeq 0$) means that it is Hermitian and
positive definite (resp., semidefinite); accordingly $A\prec 0\, (\preceq
0)$ if $-A\succ 0\, (\succeq 0)$.
For $B\in\bbC^{m\times n}$, its SVD refers to the {\em thin\/} form
$B=U\Sigma V^{\HH}$, where, with $s=\min\{m,n\}$, $U\in\STM{s}{m}$, $V\in\STM{s}{n}$, and
$\Sigma=\bbR^{s\times s}$ is diagonal with
the singular values  on the diagonal.
We use $\|B\|_2$ to denote the matrix 2-norm (i.e., the largest singular value of $B$)
and $\|B\|_{\F}:=\sqrt{\tr(B^{\HH} B)}$ the Frobenius norm.
Other notation will be explained at their first appearances.

%
%
		

\section{NPDo Approach for the Optimization Problem}\label{sec:NPDo}
In this section, we present an NPDo approach to solve the following maximization problem
\begin{subequations}\label{eq:OptOnSTM-master}
\begin{equation}\label{eq:OptOnSTM-master-a}
\max_{P\in\STM{k}{n}} \Big\{ f(P):=\sum_{i=1}^M\omega_i\tr\big([P_i^{\HH}A_iP_i]^{s_i}\big)\Big\},
\end{equation}
where
\begin{equation}\label{eq:OptOnSTM-master-b}
      \framebox{
      \parbox{12.0cm}{
      $A_i\succeq 0$ for $1\le i\le M$, positive numbers $\omega_i>0$, integers $s_i\ge 1$,  $P_i\in\bbC^{n\times k_i}$ for $1\le i\le M$ are
submatrices consisting of a few or all columns of $P$ (in particular, sharing
common columns of $P$ by different $P_i$ is allowed).
      }}
\end{equation}
\end{subequations}
We can see that each $P_i$ can be expressed as
\begin{equation}\label{eq:Pi=PJi}
P_i=PJ_i\quad\mbox{for $i=1,\dots, M$},
\end{equation}
where each $J_i\in\bbR^{k\times k_i}$ is a submatrix of $I_k$, consisting of those columns of $I_k$
with the same column indices as $P_i$ to $P$.

Problem~\eqref{eq:OptOnSTM-master} will serve
as the computational engine to our proposed method  later,
where  $s_i\in\{1,2\}$.
When all $s_i=1$, it coincides with the main problem
in~\cite{wazl:2022a} on maximizing the sum of coupled traces,
which also represents a special case of~\cite[Example 5.3]{li:2024}.
Throughout this section, we stick to the notation and assumptions
introduced with \eqref{eq:OptOnSTM-master}:
$A_i\succeq 0$
for $i=1,\dots, M$,
and $P_i$ are submatrices of $P$ as in~\eqref{eq:Pi=PJi}.

\subsection{NPDo Approach}
NPDo,  coined in \cite{li:2024},
stands for {\em nonlinear polar decomposition with orthogonal polar factor dependency\/}.
The NPDo approach was formally developed in~\cite{li:2024}, based on the
{\em NPDo Ansatz}.
While the original approach was presented for the real number field,
its extension to the complex number field is straightforward, as noted in~\cite[section~10]{li:2024}.

The optimization problem~\eqref{eq:OptOnSTM-master}
naturally fits within the NPDo framework.
Its objective function is a conic combination of terms in the form of
$\tr\big([P_i^{\HH}A_iP_i]^{s_i}\big)$
for $i=1,\dots,M$,
each of which satisfies the criteria for {\em atomic
functions} in the NPDo framework (see~\Cref{lm:NPDo:AF-PAP} below).
According to the NPDo theory developed in~\cite{li:2024},
the NPDo approach is globally convergent for
an objective function that is a convex composition of atomic functions,
making the framework directly applicable to~\eqref{eq:OptOnSTM-master}.
That said, optimization problem \eqref{eq:OptOnSTM-master} represents a special
case within the NPDo framework, allowing for several simplifications.
For clarity and self-containment, we
outline in the following a simplified realization of the NPDo approach
tailored to \eqref{eq:OptOnSTM-master}.

We begin by recalling a result from~\cite{li:2024}
on the derivatives for each individual
trace function in the objective $f$
from~\eqref{eq:OptOnSTM-master}.

\begin{lemma}[{\cite[Theorem 4.4]{li:2024}}]\label{lm:NPDo:AF-PAP}
Let $A\in\bbC^{n\times n}$ be Hermitian positive semidefinite,
$s\ge 1$  an integer,
and $Z\in\bbC^{n\times k}$ a variable.
It holds that\footnote{
     The original result was stated for real matrix variables $Z$, but in the complex case it can be understood in the same
     way as upon perturbing $Z$ to $Z+E$ for sufficiently small $\|E\|_2$ to get
     $$
     \tr\big(\big[(Z+E)^{\HH}A(Z+E)\big]^s\big)
        =\tr\big(\big[Z^{\HH}AZ\big]^s\big)
          +\Re\tr\Big(E^{\HH}\frac {\partial \tr\big(\big[Z^{\HH}AZ\big]^s\big)}{\partial Z}\Big)+O(\|E\|_2^2).
     $$
     }
$$
\frac {\partial \tr\big(\big[Z^{\HH}AZ\big]^s\big)}{\partial
Z}=2s\cdot AZ\big(Z^{\HH}AZ\big)^{s-1},
$$
and for all $Z,\,\what Z\in\bbC^{n\times m}$,  we have
\begin{subequations}\label{eq:NPDo:AF-PAP}
\begin{align}
\tr\Big(Z^{\HH}\frac {\partial \tr\big(\big[Z^{\HH}AZ\big]^s\big)}{\partial Z}\Big)
&= 2s\cdot \tr\big(\big[Z^{\HH}AZ\big]^s\big),  \label{eq:NPDo:AF-PAP-1}\\
\Re\tr\Big(\what Z^{\HH}\frac {\partial \tr\big(\big[Z^{\HH}AZ\big]^s\big)}{\partial Z}\Big)
&\le \tr\big(\big[\what Z^{\HH}A\what Z\big]^s\big)+(2s-1)\cdot \tr\big(\big[Z^{\HH}AZ\big]^s\big), \label{eq:NPDo:AF-PAP-2}
\end{align}
where $\Re(\cdot) $ extracts the real part of a complex number.
\end{subequations}
\end{lemma}

By~\Cref{lm:NPDo:AF-PAP}, we
quickly obtain an expression for the derivative of $f$
from~\eqref{eq:OptOnSTM-master} as
\begin{subequations}\label{eq:scrH-NPDo}
\begin{align}
\scrH(P):=\frac {\partial f(P)}{\partial P}
    &=\sum_{i=1}^M\omega_i\frac {\partial \tr\big([P_i^{\HH}A_iP_i]^{s_i}\big)}{\partial P_i}J_i^{\T}
        \label{eq:scrH-NPDo-1} \\
    &=\sum_{i=1}^M2s_i\omega_iA_iP_i(P_i^{\HH}A_iP_i)^{s_i-1}J_i^{\T}\in\bbC^{n\times k}. \label{eq:scrH-NPDo-2}
\end{align}
\end{subequations}
Moreover, by \eqref{eq:NPDo:AF-PAP}
--
the defining properties for {\em atomic functions} of NPDo as introduced
in~\cite{li:2024}
--
we readily establish the so-called NPDo Ansatz for the objective
$f(\cdot)$, as presented in~\Cref{thm:f(P)-Ansatz}.
While this theorem can also be viewed as a corollary of
\cite[Theorem~5.1]{li:2024},
we present it here in a different form and
include a proof for its simplicity and completeness.


\begin{theorem}[{\cite[Theorem~5.1]{li:2024}}]\label{thm:f(P)-Ansatz}
Let $\scrH(P)$ denote the derivatives of $f(P)$ as in~\eqref{eq:scrH-NPDo}.
For all $P,\,\what P\in\bbC^{n\times k}$,
\begin{equation}\label{eq:f(P)-Ansatz}
f(\what P)- f(P) \geq \Re\tr\Big(\what P^{\HH}\scrH(P)\Big) - \tr\Big(P^{\HH}\scrH(P)\Big).
\end{equation}
\end{theorem}

\begin{proof}
By~\eqref{eq:scrH-NPDo-1}, we get
\begin{align*}
\tr(\what P^{\HH}\scrH(P))
  &= \sum_{i=1}^M\omega_i\tr\Big(\what P^{\HH}\frac {\partial \tr\big([P_i^{\HH}A_iP_i]^{s_i}\big)}{\partial P_i}J_i^{\T}\Big)
  &= \sum_{i=1}^M\omega_i\tr\Big(\what P_i^{\HH}\frac {\partial \tr\big([P_i^{\HH}A_iP_i]^{s_i}\big)}{\partial P_i}\Big),
\end{align*}
where the last equality follows from the identity
$\tr(XY) = \tr(YX)$, with $Y=J_i^{\T}$,
and
$\what P_i = \what PJ_i$ is a submatrix of $\what P$
that selects the same columns as $P_i$ from $P$.
Hence, by~\eqref{eq:NPDo:AF-PAP}, we get
\begin{align*}
 \tr(P^{\HH}\scrH(P))
  &= \sum_{i=1}^M\omega_i \cdot 2s_i\tr\big([P_i^{\HH}A_iP_i]^{s_i}\big), \\ 
\Re\tr(\what P^{\HH}\scrH(P))
  &\le\sum_{i=1}^M\omega_i\Big[\tr\big([\what P_i^{\HH}A_i\what P_i]^{s_i}\big)+(2s_i-1)\tr\big([P_i^{\HH}A_iP_i]^{s_i}\big)\Big] .
\end{align*}
It follows that
$$
\Re\tr(\what P^{\HH}\scrH(P)) - \tr(P^{\HH}\scrH(P))
\leq
\sum_{i=1}^M\omega_i \tr\big([\what P_i^{\HH}A_i\what P_i]^{s_i}\big)
- \sum_{i=1}^M\omega_i  \tr\big([P_i^{\HH}A_iP_i]^{s_i}\big),
$$
where the right hand side coincides with $f(\what P)-f(P)$,
proving~\eqref{eq:f(P)-Ansatz}.
\end{proof}

A function $f$ satisfying property~\eqref{eq:f(P)-Ansatz}
is said to satisfy the NPDo Ansatz,
in the terminology of the NPDo framework in~\cite{li:2024}.
%
Note that~\eqref{eq:f(P)-Ansatz} does not require $P,\,\what P\in\STM{k}{n}$
but holds generally for $P,\,\what P\in\bbC^{n\times k}$,
which is actually more than what we will need for the purpose of
applying the NPDo approach.
While this result may seem like a special case of~\cite[Theorem~5.1]{li:2024},
it differs in that the latter is stated specifically for $P,\,\what P\in\STM{k}{n}$
(and for the real valued matrices).
%
A direct consequence of property~\eqref{eq:f(P)-Ansatz} is the following theorem,
which concerns the KKT condition of problem~\eqref{eq:OptOnSTM-master}
and forms the foundation of the NPDo approach.

\begin{theorem}[{\cite[Theorem~3.1]{li:2024}}]\label{thm:maximizer}
Let $P_*$ be a  maximizer of \eqref{eq:OptOnSTM-master}. Then the KKT condition
\begin{equation}\label{eq:KKT}
\scrH(P)=P\Lambda, \quad
P\in\STM{k}{n}, \quad \Lambda=\Lambda^{\HH}\in\bbC^{k\times k}
\end{equation}
holds for $P=P_*$
and $\Lambda=P_*^{\HH}\scrH(P_*)\succeq 0$,
where $ \scrH(P)$
is the derivative in~\eqref{eq:scrH-NPDo}.
\end{theorem}

The KKT conditions~\eqref{eq:KKT}
--
derived from standard analysis for optimization problems
with orthogonality constraints~\cite{abms:2008,li:2024}
--
provide first-order (stationary)
condition for the optimization~\eqref{eq:OptOnSTM-master}.
The key feature of~\Cref{thm:maximizer}
is the requirement that $\Lambda\succeq 0$ at the maximizer $P_*$ of
\eqref{eq:OptOnSTM-master}.
Since $\scrH(P_*)= P_*\Lambda$ and $\Lambda\succeq 0$,
it follows that $P_*\Lambda$ is a polar decomposition of $\scrH(P_*)$,
with $P_*$ being a unitary polar factor~\cite{govl:2013}.
Because $\scrH(\cdot)$ depends on $P$ nonlinearly,
equation~\eqref{eq:KKT} with $\Lambda\succeq 0$ is called a {\em nonlinear polar decomposition with
orthogonal\footnote {Strictly speaking, the term should be ``unitary'',
	but to preserve the abbreviation ``NPDo'' in \cite{li:2024}, we keep the word ``orthogonal'' here.  }
polar factor dependency\/} (NPDo) of $\scrH(\cdot)$ \cite{li:2024}.
Note that $\scrH(P_*)$ has a unique polar decomposition if
$\rank(\scrH(P_*))=k$ \cite{li:1995},
but not if $\rank(\scrH(P_*))<k$ \cite{high:2008,li:1993b,li:2014HLA}.

\subsection{An Algorithm for NPDo}
\Cref{thm:maximizer} provides the theoretical foundation of
the NPDo approach for optimization~\eqref{eq:OptOnSTM-master}
and naturally leads to the following fixed-point-type iteration for computing a maximizer:
\begin{equation}\label{eq:scfits}
\scrH(P^{(j)})=P^{(j+1)}\Lambda_j\quad\mbox{for $j=0,1,\ldots$},
\end{equation}
where $P^{(j+1)}\in\STM{k}{n}$ and $ \Lambda_j\succeq 0$;
that is, $P^{(j+1)}\Lambda_j$ is a polar decomposition of $\scrH(P^{(j)})$.
This iterative scheme resembles the well-known Self-Consistent-Field
(SCF) iterations for nonlinear eigenvector problems
and is referred to as the NPDoSCF \cite[Algorithm 3.1]{li:2024}.  

Alternatively,
in terms of the optimization problem~\eqref{eq:OptOnSTM-master},
each SCF step in~\eqref{eq:scfits} can also be interpreted as a local optimal
search based on~\Cref{thm:f(P)-Ansatz}.
In particular, given the current iterate $P^{(j)}$,
it follows from the lower bound~\eqref{eq:f(P)-Ansatz} that,
for all $P\in\STM{k}{n}$,
\begin{equation}\label{eq:scfincrease}
	f(P) - f(P^{(j)}) \geq
	\Re\tr\left(P^{\HH}\scrH(P^{(j)})\right) -
	\tr\left((P^{(j)})^{\HH}\scrH(P^{(j)})\right).
\end{equation}
Therefore, to achieve a large increase in the objective value $f(P)$,
we aim to maximize the right hand side over
$P\in\STM{k}{n}$, leading to the next iterate
\begin{equation}\label{eq:scfopt}
	P^{(j+1)} = \argmax_{P\in \STM{k}{n}} \Re\tr(P^{\HH}\scrH(P^{(j)})).
\end{equation}
This optimization subproblem also arises in the classical
orthogonal Procrustes problem in matrix analysis,
whose optimal solution is given by the
unitary polar factor of $\scrH(P^{(j)})$ (see, e.g., \cite{govl:2013}, \cite[Lemma~B.9]{li:2024}).
Thus,~\eqref{eq:scfopt} recovers the SCF iteration in~\eqref{eq:scfits}.

The optimization interpretation~\eqref{eq:scfopt} also yields a
desirable monotonicity property that underpins the global convergence of the
iteration~\eqref{eq:scfits}:
\begin{equation}
	f(P^{(j+1)})\geq f(P^{(j)}) \quad \text{for $j=0,1,\dots$},
\end{equation}
which follow from the fact that the right hand side of~\eqref{eq:scfincrease} with $P=P^{(j+1)}$
is always non-negative due to~\eqref{eq:scfopt}.
Moreover,
the inequality $f(P^{(j+1)})> f(P^{(j)})$ is strict
unless $\scrH(P^{(j)})=P^{(j)}\Lambda_j$ with $\Lambda_j\succeq 0$,
in which case $P^{(j)}$ already satisfies the fixed-point equation;
see also~\cite[Lemma~B.9]{li:2024}.
This monotonicity immediately implies the following global convergence result:
the sequence $\{f(P^{(j)})\}_{j=0}^{\infty}$
is monotonically increasing and convergent,
and any accumulation point $P_*$ of $\{P^{(j)}\}_{j=0}^{\infty}$
satisfies the KKT conditions~\eqref{eq:KKT} and hence yields a polar decomposition of $\scrH(P_*)$.
This result follows from the general convergence
theorems~\cite[Theorems~3.2~and~3.3]{li:2024},
and we omit its formal statement here for brevity.

\begin{algorithm}[t]
\caption{The NPDo approach for solving \eqref{eq:OptOnSTM-master}.}
\label{alg:master-NPDo}
\begin{algorithmic}[1]
\REQUIRE
    Hermitian $A_i\in\bbC^{n\times n}$
	satisfying $A_i\succeq 0$,
	scalars $\omega_i\geq 0$,
	and integers $s_i\geq 1$,
	for $i=1,\dots,M$,
	and an initial $P^{(0)}\in\STM{k}{n}$;
\ENSURE  an approximate maximizer of \eqref{eq:OptOnSTM-master}.
\FOR{$j=0,1,\ldots$ until convergence}
    \STATE compute $H_j=\scrH(P^{(j)})\in\bbC^{n\times k}$ by \eqref{eq:scrH-NPDo})
	and its thin SVD: $H_j=U_j\Sigma_jV_j^{\HH}$;
    \STATE $P^{(j+1)}=U_jV_j^{\HH}\in\STM{k}{n}$,
           the unitary polar factor of  $\scrH(P^{(j)})$;
\ENDFOR
\RETURN the last $P^{(j)}$.
\end{algorithmic}
\end{algorithm}

A computational process based on~\eqref{eq:scfits} is summarized
in~\Cref{alg:master-NPDo},
which is essentially  NPDoSCF~\cite[Algorithm 3.1]{li:2024} specialized to the current case.
A couple of comments on the implementation
are in order.
\begin{enumerate}[(i)]
  \item Line 2 requires computing the thin SVD of $\scrH(P)$,
	  which can be expensive when $n$ is large and $k$ is comparable to $n$.
        Fortunately, for a large-scale problem, where $n$ is very large,
		it is often sufficient to compute only a few most significant
		common principal components,
        i.e., the optimizer $P$ typically involves a small number $k$ of
		columns.
		In such cases, the thin SVD\footnote {
			The thin SVD, $H_j=U_j\Sigma_jV_j^{\HH}$ with
			$U_j\in\STM{k}{n}$ and $V_j\in\STM{k}{k}$,
            can be computed in three steps:
			(1) compute a thin QR decomposition $H_j=WR$ with $R\in\bbC^{k\times k}$,
			then (2) the SVD of $R$: $R=\what U\Sigma_j V_j^{\HH}$;
			(3) and finally $H_j=(W\what U)\Sigma_j V_j^{\HH}$, the thin SVD of $H_j$.
            Hence the overall cost per SCF iterative step, stemming from the SVD of $R$ and three matrix products of
                an $n$-by-$k$ matrix with an $k$-by-$k$ matrix, is about $6nk^2+20k^3$ flops \cite[p.493]{govl:2013}
                which is linear in $n$ for small $k$.}
        of $H_j$ presents little computational challenges.
  \item A reasonable stopping criterion at Line 1 is
         \begin{equation}\label{eq:stop-1}
         \varepsilon_{\KKT}:=\frac {\big\|\scrH(P)-P\sym\big(P^{\HH}\scrH(P)\big)\big\|_{\F}}
                 {2\sum_{i=1}^Ms_i\omega_i\|A_i\|_{\F}\|A_i\|_2^{s_i-1}}
               \le\epsilon,
         \end{equation}
         where $\epsilon$ is a given tolerance, and
         $\sym(C):=(C+C^{\HH})/2$ takes the Hermitian part of a square matrix $C$.
		 Here, the spectral norm $\|A_i\|_2$ can be expensive to compute,
		 especially when $A_i$ is of a large size. Fortunately for  normalization purpose some rough estimate
         is good enough, e.g., replacing $\|A_i\|_2$ with $\|A_i\|_1$.
		 In practice, it may be approximated by a more accessible norm,
		 such as the matrix 1-norm $\|A_i\|_1$.
		 Another option is to apply the symmetric Lanczos method
		 \cite{demm:1997,li:2010,parl:1998}, which often
		 yields a good estimate of $\|A_i\|_2$
		 after only a few iterations~\cite{zhli:2011}.
\end{enumerate}

\subsection{Acceleration with LOCG}\label{ssec:accNPDo}
 \Cref{alg:master-NPDo} requires a thin SVD of an  $n\times k$ matrix in
 each iteration, which can be computationally expensive for large-scale problems.
When $k\ll n/3$, the algorithm can be accelerated using
subspace methods based on the locally optimal conjugate gradient (LOCG) technique.
LOCG, widely used in matrix eigenvalue computations,
has also been applied to NPDo in prior works~\cite{li:2024,wazl:2022a},
where it was shown to be highly efficient.
Below, we outline its main idea and implementation for solving
the optimization problem~\eqref{eq:OptOnSTM-master}.

LOCG is an iterative subspace search scheme that constructs the new
iterate from a three-term subspace, formed by the current and previous
iterates, together with the gradient.
Specifically, let $P\in\STM{k}{n}$ be the current approximate maximizer
of~\eqref{eq:OptOnSTM-master},
and let $P^{(-1)}\in\STM{k}{n}$ be the previous iterate.
The next iterate $P^{(1)}$, along the line of LOCG, is obtained by solving
\begin{equation}\label{eq:LOCG}
P^{(1)}=\argmax_{Y\in\STM{k}{n}}f(Y)\,\,\,\mbox{s.t.}\,\,\, \cR(Y)\subseteq\cR([P,\scrR(P),P^{(-1)}]),
\end{equation}
where $\cR(\cdot)$ denotes the subspace  spanned by the columns of a matrix, and
\begin{equation}\label{eq:R(P)}
\scrR(P):=\grad f_{|{{\STM{k}{n}}}}(P)=\scrH(P)-P\cdot\sym\big(P^{\HH}\scrH(P)\big).
\end{equation}
For the first iteration, $P^{(-1)}$ is absent,
so the subspace in~\eqref{eq:LOCG} reduces to $\cR([P,\scrR(P)])$.

To solve the optimization problem~\eqref{eq:LOCG}, let
$W\in\STM{m}{n}$
be an orthonormal basis of the search subspace:
\begin{equation}\label{eq:crw}
	\cR(W) = \cR([P,\scrR(P),P^{(-1)}]).
\end{equation}
Typically, the matrix $[P,\scrR(P),P^{(-1)}]$ has a full column rank,
yielding a dimension $m=3k$ in $W$; though rank deficiency may occur,
leading to $m<3k$.
Now, since the columns of the solution $Y$
in \eqref{eq:LOCG} lie in the subspace $\cR(W)$,
we can write
\begin{subequations}\label{eq:LOCGsub}
\begin{equation}\label{eq:LOCGsub:Y}
Y=WZ\quad\mbox{for some $Z\in\STM{k}{m}$}.
\end{equation}
Substituting $Y=WZ$ into~\eqref{eq:LOCG} yields
$P^{(1)}=WZ_{\opt}$, where $Z_{\opt}$ is obtained by solving
\begin{equation}\label{eq:LOCGsub-1}
Z_{\opt}=\argmax_{Z\in\STM{k}{m}}
\left\{\wtd f(Z) := f(WZ) \equiv
\sum_{i=1}^M\omega_i\tr\big([Z_i^{\HH}\wtd A_iZ_i]^{s_i}\big) \right\},
\end{equation}
with the coefficient matrices $\wtd A_i:=W^{\HH}A_iW\in\bbC^{m\times m}$
and variables $Z_i=ZJ_i^{\HH}$, analogous to the submatrices
$P_i=PJ_i^{\T}$ of $P$,
for $i=1,\dots,M$.
\end{subequations}

Observe that the optimization problem~\eqref{eq:LOCGsub-1}
share the same form of the objective function as the original
problem~\eqref{eq:OptOnSTM-master},
but involves smaller coefficient matrices $\{\wtd A_i\}_{i=1}^M$
of size $m\times m$, where $m\ll n$ if $k\ll n/3$.
Consequently, \Cref{alg:master-NPDo} applied to solve this
reduced problem incurs significantly lower computational
cost per iteration compared to the full-size problem,
thereby providing an acceleration.

%

\begin{algorithm}[t]
\caption{LOCG-accelerated NPDo for solving~\eqref{eq:OptOnSTM-master} (assuming $k<n/3$)}
\label{alg:JBDvLOCG}
\begin{algorithmic}[1]
\REQUIRE
    Hermitian $A_i\in\bbC^{n\times n}$
	satisfying $A_i\succeq 0$,
	scalars $\omega_i\geq 0$,
	and integers $s_i\geq 1$,
	for $i=1,\dots,M$,
	and an initial $P^{(0)}\in\STM{k}{n}$;

\ENSURE  an approximate maximizer of \eqref{eq:OptOnSTM-master}.
\STATE $P^{(-1)}=[\,]$; \% null matrix;
\FOR{$j=0,1,\ldots$ until convergence}
		   \label{i:alg:JBDvLOCG:for}
    \STATE compute
	$W\in\STM{m}{n}$
	such that $\cR(W)=\cR([P^{(j)},\scrR(P^{(j)}),P^{(j-1)}])$ as in \eqref{eq:W-compute}, where
           $\scrR(P^{(j)})$ is calculated according to \eqref{eq:R(P)};
		   \label{i:alg:JBDvLOCG:orth}
    \STATE solve \eqref{eq:LOCGsub-1} for $Z_{*}$ by \Cref{alg:master-NPDo}
	with $\wtd\scrH(\cdot)$ in \eqref{eq:wtd-scrH-NPDo} and
	with initial $Z^{(0)}$ being the first $k$ columns of $I_m$;
		   \label{i:alg:JBDvLOCG:red}
    \STATE $P^{(j+1)}=WZ_{\opt}$;
		   \label{i:alg:JBDvLOCG:pj}
\ENDFOR
		   \label{i:alg:JBDvLOCG:endfor}
\RETURN the last $P^{(i)}$.
\end{algorithmic}
\end{algorithm}

\Cref{alg:JBDvLOCG} outlines the LOCG-accelerated
version of \Cref{alg:master-NPDo},
where
\begin{equation}\label{eq:wtd-scrH-NPDo}
\wtd\scrH(Z):=\frac {\partial \wtd f(Z)}{\partial Z}
   =\sum_{i=1}^M2s_i\omega_i\wtd A_iZ_i(Z_i^{\HH}\wtd
   A_iZ_i)^{s_i-1}J_i^{\T}\in\bbC^{m\times k}.
\end{equation}
A few  comments on the implementation of \Cref{alg:JBDvLOCG} are in order.
\begin{enumerate}[(i)]

  \item At Line~\ref{i:alg:JBDvLOCG:for},
	  we can use the same stopping criterion as given in~\eqref{eq:stop-1}.

	\item
		At Line~\ref{i:alg:JBDvLOCG:orth}, an orthogonalization process
		is required to obtain $W$.
		By definition, the orthonormal basis matrix $W$ for the search subspace
		in~\eqref{eq:crw} can be computed
		by orthogonalizing the columns of the matrix
		$[P,\scrR(P),P^{(-1)}]$,
		e.g., using the Gram-Schmidt process.
		Since $P\in\STM{n}{k}$ is already orthonormal,
		the orthogonalization  can begin with the remaining blocks $[\scrR(P),P^{(-1)}]$.
		In MATLAB, to fully take advantage of its optimized built-in functions,
		one may set $W=[\scrR(P),P^{(-1)}]$ (or $W=\scrR(P)$ for the
		first iteration), and execute
	\begin{equation}\label{eq:W-compute}
	\framebox{
	\begin{minipage}{9cm}
	\tt
	W=W-P*(P'*W); W=orth(W);     \% 1st Gram-Schmidt \\
	W=W-P*(P'*W); W=orth(W);     \% 2nd Gram-Schmidt \\
		  W=[P,W];
	\end{minipage}
	}
	\end{equation}
	where {\tt orth} is MATLAB's  orthogonalization
	function\footnote{MATLAB's {\tt orth} is based on the thin SVD.
		Another option is to use the thin {\tt qr}: {\tt [W,$\sim$]=qr(W,0)}.}.
	The first two lines above apply the classical Gram-Schmidt orthogonalization twice
   to ensure that the columns of the resulting $W$ are orthogonal to
   those of $P$ to nearly machine precision.
   Note that in the last line, the final matrix $W$ has its first $k$ columns identical to $P$.

  \item
	  At Line~\ref{i:alg:JBDvLOCG:red}, the initial $Z^{(0)}$ is
	  always set as the first $k$ columns of $I_m$.
	  This choice aligns with the orthogonalization process~\eqref{eq:W-compute}
	  used in Line~\ref{i:alg:JBDvLOCG:orth} for constructing the
	  projection basis matrix $W$, such that the
	  first $k$ columns of $W$ coincide with $P^{(j)}$.
	  As the iteration converges, $P^{(j+1)}$ becomes increasingly close
	  to $P^{(j)}$,
	  and consequently, the optimal $Z_{\opt}\equiv W^{\HH} P^{(j+1)}$
	  is increasingly close to the first $k$ columns of $I_m$.


  \item
	  Two additional improvements for solving the reduced problem in Line~\ref{i:alg:JBDvLOCG:red} are
	  as follows.

	  First, computational savings can be made to construct the reduced
	  $\wtd\scrH(\cdot)$
	  by reusing previously computed qualities.
	  Particularly, we may compute for the next iteration
		$$
		A_{\ell}P^{(j+1)}=(A_{\ell}W)Z_{\opt}
		\quad\text{and}\quad
		[P^{(j+1)}]^{\HH}A_{\ell}P^{(j+1)}=Z_{\opt}^{\HH}(W^{\HH}A_{\ell}W)Z_{\opt}
		$$
		using the precomputed matrices
		$(A_{\ell}W)$ and $(W^{\HH}A_{\ell}W)$,
		reducing the costs to  $O(nk^2)$ and $O(k^3)$,
		respectively, instead of $O(n^2k)$ from direct computation.


		Second, adaptive error tolerance may be used  when solving the
		reduced problem~\eqref{eq:LOCGsub-1} in each iteration.
        Suppose the error tolerance $\epsilon$ is used to stop the outer-loop
		(Lines~\ref{i:alg:JBDvLOCG:for} -- \ref{i:alg:JBDvLOCG:endfor})
		according to~\eqref{eq:stop-1}.
		Then, instead of using the same $\epsilon$ in
		\Cref{alg:master-NPDo} 
		for the reduced problem~\eqref{eq:LOCGsub-1},
		one may use a larger stopping tolerance as given by, e.g.,
		a fraction (say $1/4$) of the relative residual $\varepsilon_{\KKT}$
		evaluated at the current approximation $P=P^{(j)}$.
		The rationale is that when $P^{(j)}$ is still far from the
		solution, the corresponding reduced problem~\eqref{eq:LOCGsub-1} is only
		a rough approximation of the original one,
		and solving it by~\Cref{alg:master-NPDo} to high accuracy is unnecessary.

\end{enumerate}


\section{Principal Joint  Block Diagonalization}\label{sec:NPDo4PJBD}
In this section, we introduce the principal joint block-diagonalization (\pjbd)
problem and its solution using the NPDo approach.
To begin with,
we specify the desired block structure for a matrix of size $k$-by-$k$ by
\begin{equation}\label{eq:tau-n}
\tau_k:=(k_1,\dots,k_t),
\end{equation}
a {\em sub-partition\/} of an integer $k$,
where each $k_i\ge 1$ for $i=1,\dots, t$ 
are integers, and $k=\sum_{i=1}^t k_i$.
We define the {\em $\tau_k$-block-diagonal part\/} of
a matrix $B\in\mathbb{C}^{k\times k}$ as
\begin{equation*}
    \BDiag_{\tau_k}(B)=\diag(B_{11},\dots,B_{tt}),
\end{equation*}
where $B_{ii}\in\mathbb{C}^{k_i\times k_i}$ for $i=1,\dots,t$ are diagonal blocks
of $B$ associated with the sub-partition:
\begin{equation}\label{eq:tau-n-part}
B=\kbordermatrix{ &\sss k_1 & \sss k_2 & \sss \cdots &\sss k_t \\
         \sss k_1 & B_{11} & B_{12} & \cdots & B_{1t} \\
         \sss k_2 & B_{21} & B_{22} & \cdots & B_{2t} \\
         \sss \vdots & \vdots & \vdots &  & \vdots \\
         \sss k_t & B_{t1} & B_{t2} & \cdots & B_{tt} }.
\end{equation}
If $B=\BDiag_{\tau_k}(B)$, we refer to $B$ as a {\em
$\tau_k$-block-diagonal matrix\/}.

Given Hermitian matrices $A_{\ell}\in\bbC^{n\times n}$
for $\ell=1,\dots,N$
and a sub-partition $\tau_k$,
we consider the {\em principal joint block-diagonalization}
problem 
\begin{equation}\label{eq:opt-pjbd}
\max_{P\in\STM{k}{n}} \left\{ f(P):=\sum_{\ell=1}^N\|\BDiag_{\tau_k}(P^{\HH}A_{\ell}P)\|_{\F}^2\right\}.
\end{equation}
When $P$ is square,  i.e., $k=n$,
this optimization seeks a congruence transformation that jointly
block-diagonalize the matrices $\{A_{\ell}\}_{\ell = 1}^N$
into the block structure defined by $\tau_k$;
this is a standard joint block-diagonalization ({\jbd}) problem.
When $k<n$, the problem becomes one of {\em partial} joint block-diagonalization,
where the goal is to identify a subspace of dimension $k$ to apply {\jbd},
within which the diagonal blocks are collectively as dominant as possible.
In this sense, the optimization extracts the principal joint block-diagonal components.
We stress that the notion of the {\em partial} joint block-diagonalizer
within the principal \jbd\ problem~\eqref{eq:opt-pjbd}
has attracted little attention in the literature
and constitutes an important distinction of our formulation from prior work on {\jbd} problems.

In the rest of this section, we will explain how to solve this maximization problem with the help of
the NPDo approach  in \cref{sec:NPDo} as the working engine.
Partition $P$ columnwise, according to $\tau_k$, as
\begin{equation}\label{eq:P-part'n}
P=\kbordermatrix{ &\sss k_1 &\sss k_2 &\sss \cdots &\sss k_t \\
                  & P_1 & P_2 & \cdots & P_t},
\end{equation}
or, alternatively, $P_i=PJ_i^{\T}$ where $J_i$ for $1\le i\le t$ are from partitioning $I_k$ column-wise, also according to $\tau_k$, as
\begin{equation}\label{eq:scrH-NPDo-3}
I_k=\kbordermatrix{ &\sss k_1 &\sss k_2 &\sss \cdots &\sss k_t \\
                  & J_1 & J_2 & \cdots & J_t}.
\end{equation}
It can be seen that the objective function of \eqref{eq:opt-pjbd} can be reformulated as
\begin{equation}\label{eq:pjbd:obj-tr-1}
f(P)=\sum_{\ell=1}^N\sum_{i=1}^t\|P_i^{\HH}A_{\ell}P_i\|_{\F}^2
    =\sum_{\ell=1}^N\sum_{i=1}^t\tr\big([P_i^{\HH}A_{\ell}P_i]^2\big)
\end{equation}
which is the sum of $M:=tN$ matrix traces
$\tr\big([P_i^{\HH}A_{\ell}P_i]^2\big)$ and thus has the same form as the one in \eqref{eq:OptOnSTM-master} with  all $s_i=2$ and $\omega_i=1$,  but whether all $A_{\ell}\succeq 0$ is a question
that needs to be addressed.


Before proceeding,
we point out two special cases of \eqref{eq:opt-pjbd}:
(a)~If $\tau_k = (1,\dots,1)$,
i.e., $k_i=1$ for $i=1,\dots,t$ and $t=k$,
then the problem
reduces to finding the most dominant partial joint diagonalization;
(b)~If $\tau_k = (k)$,
i.e., $k_1=k$ and $t=1$,
then it corresponds to seeking
the most dominant joint compression.

\begin{remark}\label{rk:fjd}
	As an illustrative example to build further intuition for the principal
	{\jbd} problem~\eqref{eq:opt-pjbd},
	we consider the ideal case where the Hermitian matrices $\{A_{\ell}\}_{\ell=1}^N$ are
	exactly jointly diagonalizable.
	In this case, there exits $Q\in\STM{n}{n}$, such that
	\[
    Q^{\HH} A_{\ell} Q = D_{\ell}\equiv\diag(a_{\ell;11},\dots,a_{\ell;nn})
	\]
	for $\ell=1,\dots, N$.
	By a suitable permutation, we can always assume the diagonal entries
	of $\{D_{\ell}\}_{\ell =1}^N$
	are in the  order dictated by
	\begin{equation}\label{eq:pkorder}
		d_1\geq d_2\geq \dots\geq d_n
	\quad\text{with}\quad
		d_i:= \sum_{\ell=1}^n a_{\ell; ii}^2.
	\end{equation}
	In this case, the first $k$ columns of $Q$,
	denoted by $Q_{1:k}$, naturally captures
	the $k$ most significant joint-diagonalization components.
	Those components in $Q_{1:k}$ consist of the eigenvectors of
	the top $k$ eigenvalues of the Hermitian matrix
\begin{equation}\label{eq:Atot}
A_{\tot} :=
	\sum_{\ell=1}^NA_{\ell}^2.
\end{equation}
	On the other hand, a quick verification shows that $Q_{1:k}$ is exactly a solution
	of the principal {\jd} problem~\eqref{eq:opt-pjbd}, i.e., with all $k_i=1$ and $t=k$.
	This follows from the observation that for any $P\in\STM{k}{n}$
\begin{subequations}\label{eq:rk:fjd-1}
\begin{equation}\label{eq:rk:fjd-1a}
		 f(P) =
		 \sum_{\ell=1}^N\tr\left([P^{\HH}A_{\ell} P]^2\right)
		 \leq
		 \sum_{\ell=1}^N\tr\left([P^{\HH}A_{\ell}^2 P\right)
		 = \tr\left( P^{\HH}A_{\tot} P\right),
\end{equation}
	 where the inequality is due to the fact that
\begin{equation}\label{eq:rk:fjd-1b}
\tr\left(P^{\HH}A_{\ell} P\cdot P^{\HH}A_{\ell} P \right) \leq
	 \tr\left(A_{\ell} P P^{\HH}A_{\ell}\right) = \tr\left( P^{\HH}A_{\ell}^2 P\right).
\end{equation}
\end{subequations}
It is known, by Ky Fan's trace minimization/maximization principle \cite{fan:1949,liwz:2023}, that
$\tr\left( P^{\HH}A_{\ell}^2 P\right)$ is maximized when
	 $P=Q_{1:k}$ consists of the eigenvectors associated with the $k$ largest eigenvalues of
	 $A_{\tot}$, at which the equalities in \eqref{eq:rk:fjd-1} also hold due to simultaneous
	 diagonalization of $\{A_{\ell}\}_{\ell=1}^N$ by  $Q_{1:k}$.
	 Consequently, when $\{A_{\ell}\}_{\ell=1}^N$ are approximately  jointly diagonalizable,
	 we may still expect that, approximately, the principal {\jbd} problem~\eqref{eq:opt-pjbd}
	 yields a partial diagonalizer closely aligned with those from
	 an approximate full diagonalizer $P$,
	 ordered according to~\eqref{eq:pkorder}.
	 Yet we can avoid the full {\jd} to obtain such a partial diagonalizer for numerical efficiency.
\end{remark}

\subsection{Positive Semidefinite Case}\label{sec:NPDo4PJBD_spd}
Suppose that  $A_{\ell} \succeq 0$ for $\ell=1,\dots,N$.
Via reformulation~\eqref{eq:pjbd:obj-tr-1}, 
the NPDo approach described in~\cref{sec:NPDo} applies
directly to the \pjbd\ problem~\eqref{eq:opt-pjbd}.
In particular, for \Cref{alg:master-NPDo},
the corresponding $\scrH$ of \eqref{eq:scrH-NPDo} is
\begin{equation}\label{eq:scrH-pjbd4PSD}
\scrH(P):=\frac {\partial f(P)}{\partial P}
   =4\sum_{\ell=1}^N\Big[ A_{\ell}P_1(P_1^{\HH} A_{\ell}P_1),\ldots, A_{\ell}P_t(P_t^{\HH} A_{\ell}P_t)\Big],
\end{equation}
and for \Cref{alg:JBDvLOCG}, the corresponding $\wtd\scrH$ of \eqref{eq:wtd-scrH-NPDo}
is given by
\begin{equation}\label{eq:wtd-scrH-pjbd4PSD}
\wtd\scrH(Z):=\frac {\partial\wtd f(Z)}{\partial Z}
   =4\sum_{\ell=1}^N\Big[ \wtd A_{\ell}Z_1(Z_1^{\HH} \wtd A_{\ell}Z_1),\ldots,
        \wtd A_{\ell}Z_t(Z_t^{\HH} \wtd A_{\ell}Z_t)\Big].
\end{equation}

\subsection{General Hermitian Case}\label{sec:NPDo:g}
We now consider the case when some $A_{\ell}\not\succeq 0$,
which includes either $A_{\ell}\preceq 0$ or indefinite.
In this case, the trace function  $\tr\big([P_i^{\HH}A_{\ell}P_i]^2\big)$ in~\eqref{eq:pjbd:obj-tr-1}
may no longer satisfy inequality~\eqref{eq:NPDo:AF-PAP-2} in~\Cref{lm:NPDo:AF-PAP} (i.e., not an atomic function),
and hence the NPDo approach cannot be applied directly.
To address this issue, we will shift each Hermitian matrix needs to a
positive semi-definite matrix,
and the optimization problem is reformulated accordingly to admit the use
of the NPDo approach.
The details are presented below.

To begin with, suppose we have shift parameters
$\delta_{\ell}\in\bbR$, for $\ell=1,\dots,N$, such that
\begin{equation}\label{eq:bding-Aell}
\what A_{\ell}:=A_{\ell}-\delta_{\ell}I\succeq 0
\quad\mbox{for $\ell=1,\dots, N$}.
\end{equation}
Shift $\delta_{\ell}$ ca be taken as any lower bound on the eigenvalues of $A_{\ell}$, and numerically
it can be estimated rather efficiently \cite{zhli:2011}.
Reformulating each trace function
$\tr((P_i^{\HH}A_{\ell}P_i)^2)$ as
\begin{align*}
\tr((P_i^{\HH}A_{\ell}P_i)^2)&=\tr([P_i^{\HH}\what A_{\ell}P_i]^2)+2\delta_{\ell}\tr(P_i^{\HH}\what A_{\ell}P_i)+k_i\delta_{\ell}^2\\
   &=\tr([P_i^{\HH}\what A_{\ell}P_i]^2)+2\delta_{\ell}\tr(P_i^{\HH} A_{\ell}P_i)-k_i\delta_{\ell}^2,
\end{align*}
we obtain
\begin{align}
f(P)&=\sum_{\ell=1}^N\sum_{i=1}^t\tr\left([P_i^{\HH}A_{\ell}P_i]^2\right) \nonumber \\
    &=
    \sum_{\ell=1}^N\sum_{i=1}^t\tr\left([P_i^{\HH}\what A_{\ell}P_i]^2\right)
      +2\tr(P^{\HH} B P)
       -k\sum_{\ell=1}^N\delta_{\ell}^2, \label{eq:trans2PSD}
\end{align}
where
\begin{equation}\label{eq:B=sum}
B=\sum_{\ell=1}^N\delta_{\ell}\, A_{\ell}.
\end{equation}
The expression of $f(P)$ in~\eqref{eq:trans2PSD} resembles
the objective function in~\eqref{eq:OptOnSTM-master},
except that it has an extra constant term and $B$ is not necessarily positive semidefinite.
To address the indefiniteness of $B$, we treat the cases $k=n$ and $k<n$ separately.

\smallskip\noindent
{\bf Case $k=n$.} Now that $P\in\STM{n}{n}$, we have
\begin{equation}\label{eq:sum2whole}
2\tr(P^{\HH} B P)
=
2 \tr( B)
=
2 \sum_{\ell=1}^N\delta_{\ell}\tr(A_{\ell}),
\end{equation}
a constant. It follows from \eqref{eq:trans2PSD} that
\begin{equation}\label{eq:trans2PSD:k=n}
f(P)=\underbrace{\sum_{\ell=1}^N\sum_{i=1}^t\tr\left([P_i^{\HH}\what A_{\ell}P_i]^2\right)}_{=:\what f(P)}
      +\underbrace{2 \sum_{\ell=1}^N\delta_{\ell}\tr(A_{\ell})-k\sum_{\ell=1}^N\delta_{\ell}^2}_{\mbox{constant}}.
\end{equation}
Consequently, optimization problem \eqref{eq:pjbd:obj-tr-1} is
equivalent to
\begin{equation}\label{eq:opt-pjbd:trans2PSD:k=n}
\max_{P\in\STM{n}{n}}\, \what f(P).
\end{equation}
Objective function $\what f(P)$ takes the same form as $f(P)$
in~\eqref{eq:pjbd:obj-tr-1},
with all trace terms $\tr\big([P_i^{\HH}\what A_{\ell}P_i]^2\big)$
involving $\what A_{\ell}\succeq 0$.
Thus, the problem reduces to the positive semidefinite case as
discussed in subsection~\ref{sec:NPDo4PJBD_spd},
and the NPDo approach applies.

\smallskip\noindent
{\bf Case $k<n$.}
For the case, the term $\tr(P^{\HH} B P)$ is no longer a constant as in \eqref{eq:sum2whole}.
Similarly, we shift matrix $B$, using $\delta_0\in\bbR$, to
a positive semidefinite one:
$$
\what B:=B-\delta_0 I\succeq 0.
$$
It then follows from \eqref{eq:trans2PSD} that
\begin{equation}\label{eq:trans2PSD'}
f(P)=\underbrace{\sum_{\ell=1}^N\sum_{i=1}^t\tr\left([P_i^{\HH}\what A_{\ell}P_i]^2\right)
      +2\tr(P^{\HH}\what B P)}_{=:\what f(P)}
      +\underbrace{2k\delta_0-k\sum_{\ell=1}^N\delta_{\ell}^2}_{\mbox{constant}}.
\end{equation}
Again, by dropping the constant portion,
we arrive at an equivalent maximization problem to \eqref{eq:pjbd:obj-tr-1}:
\begin{equation}\label{eq:opt-pjbd:trans2PSD:k<n}
\max_{P\in\STM{k}{n}}\, \what f(P),
\end{equation}
Now the objective function $\what f(P)$ takes the same form as $f(P)$
in \eqref{eq:OptOnSTM-master}, i.e., consisting of trace terms in the
form of $\tr([P_i^{\HH}AP_i]^{s_i})$ with $A\succeq 0$ and $s_i=1$ or $2$.
Consequently, the NPDo approach applies to the optimization
problem~\eqref{eq:opt-pjbd:trans2PSD:k<n}.
In particular,
for \Cref{alg:master-NPDo}, the corresponding $\scrH$ is
\begin{equation}\label{eq:scrH-trans2PSD:k<n}
\what\scrH(P):=\frac {\partial\what f(P)}{\partial P}
   =4\sum_{\ell=1}^N\Big[\what A_{\ell}P_1(P_1^{\HH}\what A_{\ell}P_1),\ldots,\what A_{\ell}P_t(P_t^{\HH}\what A_{\ell}P_t)\Big]+4\what BP,
\end{equation}
and for \Cref{alg:JBDvLOCG}, the corresponding $\wtd\scrH$ of \eqref{eq:wtd-scrH-NPDo}
is given by
\begin{equation}\label{eq:wtd-scrH-pjbd4PSD:k<n}
\wtd\scrH(Z):=\frac {\partial\wtd f(Z)}{\partial Z}
   =4\sum_{\ell=1}^N\Big[ \wtd A_{\ell}Z_1(Z_1^{\HH} \wtd A_{\ell}Z_1),\ldots,
        \wtd A_{\ell}Z_t(Z_t^{\HH} \wtd A_{\ell}Z_t)\Big]
		+ 4\what BZ.
\end{equation}

\subsection{Complete Algorithm}
From the previous discussions, the \pjbd\
problem~\eqref{eq:opt-pjbd} can be reduced
to an optimization problem of the form~\eqref{eq:OptOnSTM-master}
and solved using the NPDo approach via
\Cref{alg:master-NPDo} or \Cref{alg:JBDvLOCG},
with appropriate matrix-valued functions $\scrH$ or $\wtd\scrH$
for the partial derivatives, depending on the cases:
\begin{enumerate}[(1)]
  \item When all $A_{\ell}\succeq 0$,
	  	either known a priori or detected by $\delta_{\ell}\ge 0$ with
		some sharp estimation $\delta_{\ell}$ of the lower bound
        of the smallest eigenvalues of $A_{\ell}$,
		then the optimization problem~\eqref{eq:opt-pjbd} takes the form
		of~\eqref{eq:OptOnSTM-master} directly via the
		reformulation in~\eqref{eq:pjbd:obj-tr-1}.

  \item Otherwise, 
	  using the shifted matrices $\what
	  A_{\ell}\succeq 0$ as in \eqref{eq:bding-Aell}, two subcases arise:
        \begin{enumerate}[(a)]
          \item If $k=n$, then problem \eqref{eq:opt-pjbd} reduces to~\eqref{eq:opt-pjbd:trans2PSD:k=n}
              with the objective function defined in \eqref{eq:trans2PSD:k=n};
          \item If $k<n$, compute $B$ as in~\eqref{eq:B=sum}
			  and $\what B=B-\delta_0I$, where $\delta_0$ is some sharp estimate of the
			  lower bound of the smallest eigenvalue of $B$.
			  Then problem~\eqref{eq:opt-pjbd} reduces to
			  \eqref{eq:opt-pjbd:trans2PSD:k<n}
              with the objective function defined in \eqref{eq:trans2PSD'}.
        \end{enumerate}
        The objective functions $\what f(P)$ for both subcases take the form of \eqref{eq:OptOnSTM-master}.
\end{enumerate}
Based on these cases, the complete algorithm for \pjbd\ via solving~\eqref{eq:opt-pjbd} is summarized in
\Cref{alg:master-NPDo:g}.

\begin{algorithm}[t]
\caption{\pjbd: Principal Joint Block Diagonalization.} \label{alg:master-NPDo:g}
\begin{algorithmic}[1]
\REQUIRE Hermitian $A_{\ell}\in\bbC^{n\times n}$
for $\ell=1,\dots, N$,
         sub-partition $\tau_k$,
         an initial $P^{(0)}\in\STM{k}{n}$.
\ENSURE  an approximate maximizer of \eqref{eq:opt-pjbd}.
\STATE set ${\tt is_{SPD}}=1$ if it is known that all $A_{\ell}\succeq 0$ upon entry, and $0$ otherwise;
\STATE if ${\tt is_{SPD}}=0$, estimate a 
lower bound $\delta_{\ell}$ of the smallest eigenvalues of $A_{\ell}$ for all $\ell$;
\STATE if all $\delta_{\ell}\ge 0$, then reset ${\tt is_{SPD}}=1$;
\IF{${\tt is_{SPD}}=1$}
   \STATE call \Cref{alg:master-NPDo}
          or \Cref{alg:JBDvLOCG} to solve \eqref{eq:opt-pjbd} upon noting \eqref{eq:pjbd:obj-tr-1};
\ELSE
   \STATE compute $\what A_{\ell}=A_{\ell}-\delta_{\ell} I$;
   \IF{$k=n$}
       \STATE call \Cref{alg:master-NPDo} or \Cref{alg:JBDvLOCG} to
	   solve \eqref{eq:opt-pjbd:trans2PSD:k=n};
   \ELSE
       \STATE compute $B$ by \eqref{eq:B=sum}, estimate a
	   lower bound $\delta_0$ of the smallest eigenvalue of $B$, and set $\what B=B-\delta_0 I$;
       \STATE call \Cref{alg:master-NPDo}
              or \Cref{alg:JBDvLOCG} to solve
			  \eqref{eq:opt-pjbd:trans2PSD:k<n}; 
   \ENDIF
\ENDIF
\RETURN the computed $P^{(j)}$ by \Cref{alg:master-NPDo} or \Cref{alg:JBDvLOCG}, whichever is called.
\end{algorithmic}
\end{algorithm}

As a practical consideration,
when it is not known in advance whether all Hermitian matrices
$A_{\ell}$ are positive semidefinite,
\Cref{alg:master-NPDo:g} requires an efficient method to determine
which matrices satisfy $A_{\ell}\succeq 0$.
This is done by estimating a lower bound
$\delta_{\ell}\leq \lambda_{\min}(A_{\ell})$ for
the smallest eigenvalue  of the Hermitian matrix $A_{\ell}$,
using, e.g., the inexpensive sharp estimate described in~\cite{zhli:2011}.
Clearly,  if $\delta_{\ell}\geq 0$, then $A_{\ell}\succeq 0$.
However, we should mention that if the lower bound $\delta_{\ell}< 0$,
then it is impossible to conclude
whether the smallest eigenvalue $\lambda_{\min}(A_{\ell})$ is positive,
and thus $A_{\ell}$ may or may not be positive semidefinite.
In such cases, we can always shift the matrix to a positive
semidefinite $\what A_{\ell} = A-\delta I \succeq 0$ and proceed with the
equivalent formulation of $f(P)$ in~\eqref{eq:trans2PSD}, allowing
the NPDo approach to be applied as described in~\cref{sec:NPDo:g}.
Therefore, \Cref{alg:master-NPDo:g} tolerates the cases where
$A_{\ell}\succeq 0$ is misclassified due to a negative lower bound $\delta_{\ell} <0$.

\section{Numerical Demonstrations}\label{sec:egs}

We present numerical experiments to illustrate the effectiveness of
our proposed NPDo approach for \pjbd,
which include two particular variants:
\begin{itemize}
	\item
	\Cref{alg:master-NPDo:g} combined with \Cref{alg:master-NPDo}, referred to as NPDo, and
	\item
	\Cref{alg:master-NPDo:g} combined with \Cref{alg:JBDvLOCG}, referred to as accNPDo.
\end{itemize}
We make sure to use the same initial $P^{(0)}$ is used for each testing set $\{A_{\ell}\}_{\ell=0}^N$.
For comparison, in the special cases where each block size is 1-by-1
-- that is,  joint diagonalization (\jd) or principal joint diagonalization (\pjd) (when $k<n$)
--
we also include two existing methods as mentioned in~\cref{sec:intro}:
\begin{itemize}
  \item Jacobi's method as implemented by Cardoso and Souloumiac\footnote{\tt https://www2.als.lbl.gov/als\_physics/csteier/uspas12/compuclass/tuesday/ICA/mfile/}, and
  \item the FG-algorithm implemented by W. Gander\footnote{\tt
	  https://www.unige.ch/$\sim$gander/FG.php}.
\end{itemize}
As a common basis to measure accuracy of computed solutions from each method, we will report
their normalized KKT residuals:
\begin{equation}\label{eq:stop-JBD}
\varepsilon_{\KKT}:=\frac {\big\|\scrH(P)-P\sym\big(P^{\HH}\scrH(P)\big)\big\|_{\F}}
                 {4\sum_{\ell=1}^N\|A_{\ell}\|_{\F}\|A_{\ell} \|_2},
\end{equation}
where the normalization quantity in the denominator differs from the one in \eqref{eq:stop-1}, if straightforwardly applied,
because of the special form of $\scrH(P)$ in \eqref{eq:scrH-pjbd4PSD}, for which
$$
\Big\|\Big[ A_{\ell}P_1(P_1^{\HH} A_{\ell}P_1),\ldots, A_{\ell}P_t(P_t^{\HH} A_{\ell}P_t)\Big]\Big\|_{\F}
   \le\|A_{\ell}\|_{\F}\|A_{\ell}\|_2.
$$
Note that the FG-algorithm applies only to symmetric positive semidefinite matrices.
Therefore, in experiments involving FG, we ensure that each $A_{\ell}\succeq 0$ and is real.


All experiments are carried out within the MATLAB environment (MATLAB R2022) on a Dell Precision 3660 desktop with an
Intel i9 processor (3200 Mhz), 32 GB memory, running Microsoft Windows 11 Enterprise.

Both Jacobi's method~\cite{bubm:1993,caso:1996} and the FG-algorithm~\cite{flga:1986}
compute full \jd\ (i.e., $k=n$ and all $k_i=1$).
The objective for Jacobi's method is still the one
in~\eqref{eq:opt-pjbd}, but with all $k_i=1$ and $k=t=n$,
namely
\begin{equation}\label{eq:objjacobi}
f(P):=\sum_{\ell=1}^N\|\diag(P^{\HH}A_{\ell}P)\|_{\F}^2,
\end{equation}
where $\diag(\cdot)$ extracts the diagonal of a matrix.
Jacobi's method relies on accumulative planar rotations that sweep through all
$n(n-1)/2$ off-diagonal entries of the matrices during each outer iteration.
The objective for the FG-algorithm is, instead,
\begin{equation}\label{eq:tot-devi:det}
f(P):=\prod_{\ell=1}^N\left[\frac {\det(\diag(P^{\HH}A_{\ell}P))}{\det(A_{\ell})}\right]^{\omega_{\ell}},
\end{equation}
where $\omega_{\ell}\geq 0$ are given weights. Notice that all denominators $\det(P^{\HH}A_{\ell}P)$ are actually constants and they are there more for mathematical consideration than anything else, for example, each fraction-factor in
\eqref{eq:tot-devi:det} is no bigger than $1$ and is $1$ if $P^{\HH}A_{\ell}P$ is diagonal, due to
the well-known Hadamard's inequality \cite[p.506]{hojo:2013}.
Also using planar rotations, the FG-algorithm iteratively solves the equations derived from the first order optimality conditions, in the form of $n(n-1)/2$ scalar equations \cite{flur:1984}.
For large $n$, both methods can be prohibitively expensive,
particularly when the matrices are far from being jointly diagonalizable,
where they typically require a large number of outer iterations.
Consequently, both Jacobi's method and the FG-algorithm
are only suitable for small-scale problems
(e.g., $n$ on the order of a couple of hundreds or smaller).

\smallskip\noindent
{\bf Problem Generation.}
For our experiments, we generate testing Hermitian matrices
$\{A_{\ell}\}_{\ell=1}^N$ of size $n$-by-$n$ (real or complex)
that are either
{\em not jointly diagonalizable\/} or {\em approximately jointly diagonalizable}.

For not jointly block-diagonalizable $\{A_{\ell}\}_{\ell=1}^N$  with all $A_{\ell}\succeq 0$, we simply do for $\ell=1,2,\ldots,N$
\begin{equation}\label{eq:NJD-B}
B_{\ell}={\tt randn}(n)
   \quad\mbox{or}\quad {\tt randn}(n)+{\tt 1i*}{\tt randn}(n),
\end{equation}
depending on whether {\em real\/} or {\em complex\/} matrices are used, and then
\begin{equation}\label{eq:NJD}
A_{\ell}=B_{\ell}^{\HH}B_{\ell}.
\end{equation}
For approximately jointly block-diagonalizable $\{A_{\ell}\}_{\ell=1}^N$, we first
generate a random orthogonal (or unitary) matrix $Q$ by
\begin{equation}\label{eq:aJD-1}
Q={\tt orth}({\tt randn}(n))
\quad\mbox{or}\quad
{\tt orth}({\tt randn}(n)+{\tt 1i*}{\tt randn}(n)),
\end{equation}
again depending on whether {\em real\/} or {\em complex\/} matrices are used, and then,
for $\ell=1,\dots,N$, we construct, for indefinite $A_{\ell}$,
\begin{align}
A_{\ell}&=Q^{\HH}D_{\ell}Q+\eta(B_{\ell}+B_{\ell}^{\HH}), \label{eq:aJD-2b}
\end{align}
or, for positive semidefinite $A_{\ell}$,
\begin{equation}\label{eq:JD-noJD-2}
A_{\ell}=Q^{\HH}D_{\ell}Q+\eta( B_{\ell}^{\HH}B_{\ell}),
\end{equation}
where $D_{\ell}$ is either diagonal or block-diagonal (to be specified later),
and $\eta$ is a parameter that controls the degree of being approximately jointly block-diagonalizable.
In our test, we vary $\eta$ over a range of values, e.g.,  $10^0\sim 10^{-3}$.
Evidently in \eqref{eq:JD-noJD-2}, making $D_{\ell}\succeq 0$ and $\eta\ge 0$ ensures $A_{\ell}\succeq 0$.

Dependent on whether {\em real\/} or {\em complex\/} matrices are used, associated with each testing set
$\{A_{\ell}\}_{\ell=1}$, we also generate a random initial
\begin{equation}\label{eq:P0-init}
P^{(0)}={\tt orth}({\tt randn}(n,k))
\quad\mbox{or}\quad
{\tt orth}({\tt randn}(n,k)+{\tt 1i*}{\tt randn}(n,k))
\end{equation}
to be used by both NPDo and accNPDo to begin with.

\setlength{\tabcolsep}{4pt}
\renewcommand{\arraystretch}{1.4}
\begin{table}[t]
\caption{Performance statistics for full \jd}\label{tbl:fullJD}
\centerline{\small
\begin{tabular}{|c|c|c|c|c|c|c|c|c|c|c|}
  \hline
&\multirow{2}{*}{$n$}  & \multicolumn{3}{c|}{CPU} & \multicolumn{3}{c|}{$\varepsilon_{\KKT}$ } & \multicolumn{3}{c|}{Obj} \\ \cline{3-11}
  && NPDo & Jacobi & FG& NPDo & Jacobi & FG& NPDo & Jacobi & FG \\ \hline
\multirow{4}{*}{\rotatebox{90}{real}}&\multirow{2}{*}{100}
	& $6.1{\scriptstyle (+1)}$ & $\mathbf{4.6{\scriptstyle (+1)}}$ & $2.6{\scriptstyle (+3)}$
	& $1.0{\scriptstyle (-8)}$ & $1.4{\scriptstyle (-8)}$ & $3.6{\scriptstyle (-3)}$
	& $1.4{\scriptstyle (+7)}$ & $1.4{\scriptstyle (+7)}$ & $1.2{\scriptstyle (+7)}$\\ \cline{3-11}
	&& $1.2{\scriptstyle (+0)}$ & $\mathbf{2.4{\scriptstyle (-1)}}$ & $2.4{\scriptstyle (+0)}$
	& $9.8{\scriptstyle (-9)}$ & $1.5{\scriptstyle (-8)}$ & $4.5{\scriptstyle (-4)}$
    & $3.4{\scriptstyle (+4)}$ & $3.4{\scriptstyle (+4)}$ & $3.4{\scriptstyle (+4)}$\\ \cline{2-11}
&\multirow{2}{*}{200}
    & $\mathbf{3.5{\scriptstyle (+2)}}$ & $6.8{\scriptstyle (+2)}$ & $3.4{\scriptstyle (+4)}$
	& $1.0{\scriptstyle (-8)}$ & $1.4{\scriptstyle (-8)}$ & $2.4{\scriptstyle (-3)}$
	& $1.1{\scriptstyle (+8)}$ & $1.1{\scriptstyle (+8)}$ & $9.1{\scriptstyle (+7)}$\\ \cline{3-11}
    && $5.1{\scriptstyle (+0)}$ & $\mathbf{1.6{\scriptstyle (+0)}}$ & $2.7{\scriptstyle (+1)}$
	& $1.0{\scriptstyle (-8)}$ & $1.8{\scriptstyle (-8)}$ & $5.4{\scriptstyle (-4)}$
	& $7.0{\scriptstyle (+4)}$ & $7.0{\scriptstyle (+4)}$ & $7.0{\scriptstyle (+4)}$\\
\hline
\multirow{4}{*}{\rotatebox{90}{complex}}&\multirow{2}{*}{100}
	& $\mathbf{8.3{\scriptstyle (+1)}}$ & $3.2{\scriptstyle (+2)}$ & --
	& $1.0{\scriptstyle (-8)}$ & $1.4{\scriptstyle (-8)}$ & --
	& $5.5{\scriptstyle (+7)}$ & $5.5{\scriptstyle (+7)}$ & -- \\ \cline{3-11}
	&& $2.7{\scriptstyle (+0)}$ & $\mathbf{2.0{\scriptstyle (+0)}}$ & --
	& $1.0{\scriptstyle (-8)}$ & $1.6{\scriptstyle (-8)}$ & --
	& $3.5{\scriptstyle (+4)}$ & $3.5{\scriptstyle (+4)}$ & --\\ \cline{2-11}
&\multirow{2}{*}{200}
	& $\mathbf{1.0{\scriptstyle (+3)}}$ & $1.1{\scriptstyle (+4)}$ & --
	& $1.4{\scriptstyle (-7)}$ & $9.7{\scriptstyle (-7)}$ & --
	& $4.3{\scriptstyle (+8)}$ & $4.3{\scriptstyle (+8)}$ & -- \\ \cline{3-11}
	&& $\mathbf{9.8{\scriptstyle (+0)}}$ & $1.7{\scriptstyle (+1)}$ & --
	& $1.0{\scriptstyle (-8)}$ & $1.7{\scriptstyle (-8)}$ & --
	& $7.5{\scriptstyle (+4)}$ & $7.5{\scriptstyle (+4)}$ & -- \\ \cline{1-11}
\multicolumn{11}{l}{{\scriptsize  * Each $x.y{\scriptstyle (\pm z)}$ stands for $x.y\times 10^{\pm z}$; best CPU in each row is in {\bf boldface}.}}
\end{tabular}
}
\end{table}

\subsection{Experiments on Full \jd}\label{ssec:full-jd}
We consider full {\jd} for test matrices of relatively small sizes,
$n=100$ and $200$.  The performance of the proposed NPDo approach
(via \Cref{alg:master-NPDo})
is compared with that of the existing Jacobi's method and the FG-algorithm.
Since the FG-algorithm is involved, all test matrices are generated to satisfy $A_{\ell}\succeq 0$.
Particularly, for the not-jointly-diagonalizable case, we generate $A_{\ell}\succeq 0$
using~\eqref{eq:NJD} with \eqref{eq:NJD-B},
and for the approximately jointly diagonalizable case,
we use \eqref{eq:aJD-2b} with $\eta=10^{-3}$ and
\begin{equation}\label{eq:aJD-3}
D_{\ell}=
{\tt diag}(10\cdot{\tt rand}(n,1)).
\end{equation}
Both real and complex cases are considered.

\begin{figure}[t]
{\centering
\begin{tabular}{lcc}
& approximately jointly diagonalizable & not jointly diagonalizable \\ [1ex]
\rotatebox{90}{\hspace*{1.8cm}real} &
\resizebox*{0.44\textwidth}{0.21\textheight}{\includegraphics{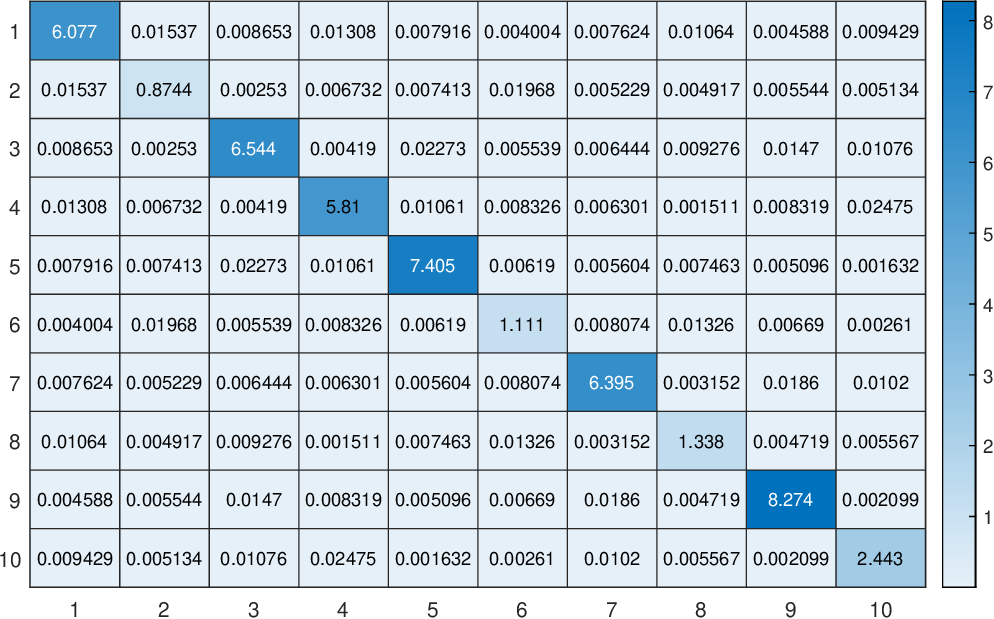}}
  & \resizebox*{0.44\textwidth}{0.21\textheight}{\includegraphics{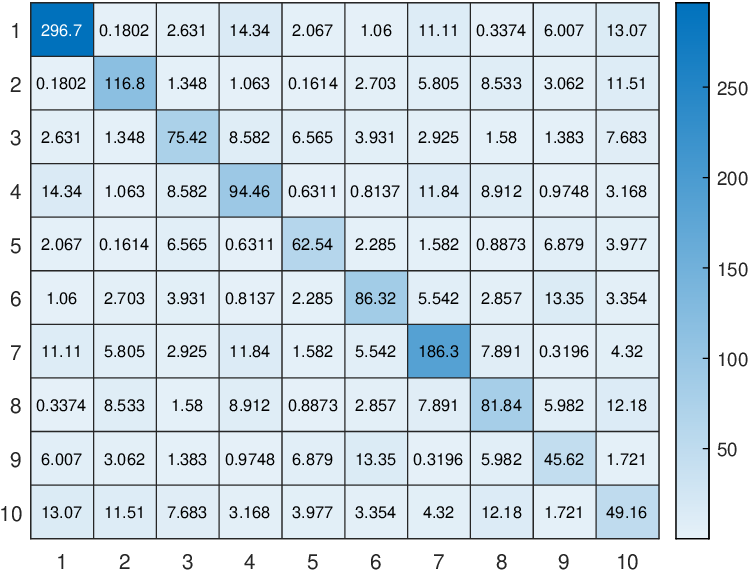}} \\ [1ex]
\rotatebox{90}{\hspace*{1.8cm}complex} &
\resizebox*{0.44\textwidth}{0.21\textheight}{\includegraphics{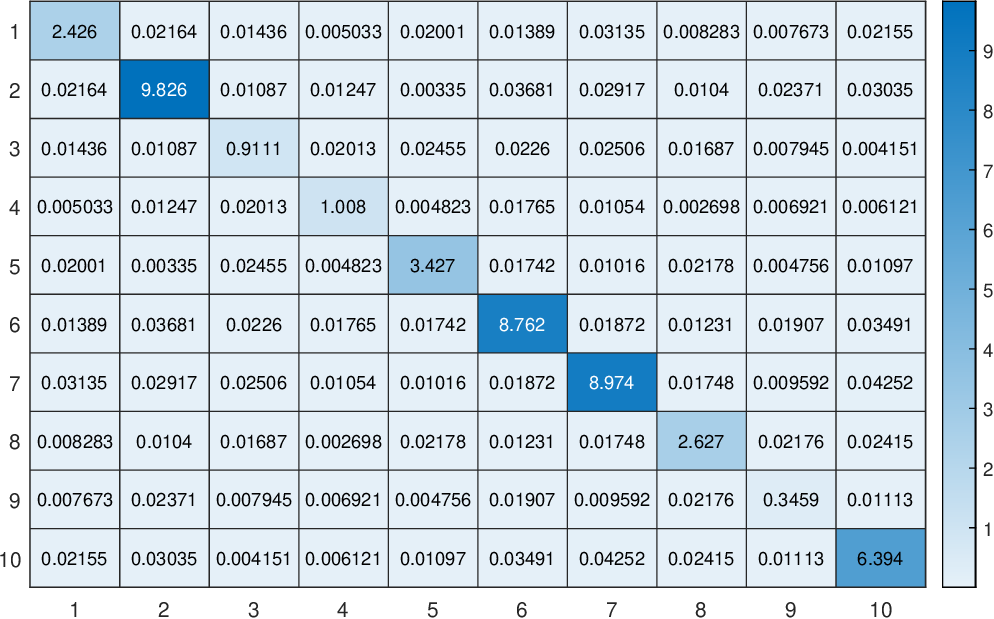}}
  & \resizebox*{0.44\textwidth}{0.21\textheight}{\includegraphics{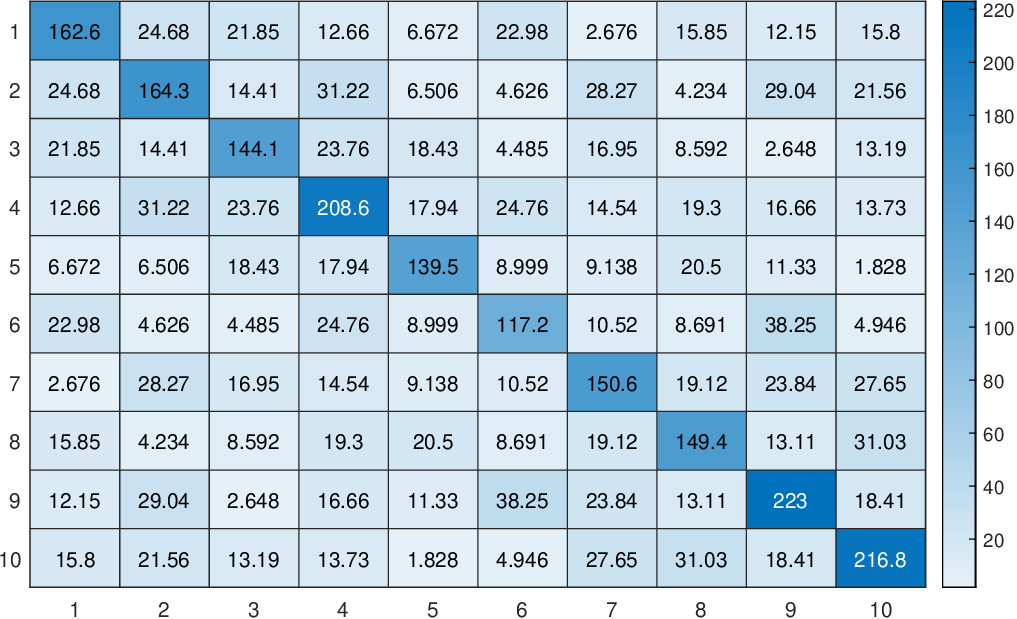}} \\
\end{tabular}\par
}
\vspace{-0.15 cm}
\caption{\small Heatmap of the leading $10$-by-$10$ principal submatrix of converged $P^{\HH}A_1P$, by NPDo, corresponding to its $10$ largest
diagonal entries for $n=100$ with the options: \{real or complex\} times \{jointly diagonalizable or not\}.
  }
\label{fig:fullJD-heat-A1}
\end{figure}

Table~\ref{tbl:fullJD} reports performance statistics for the three
approaches, including CPU time (in seconds), normalized KKT residual $\varepsilon_{\KKT}$ in \eqref{eq:stop-JBD},
and the objective value at convergence.
For each value of $n$, there are two rows: 
the first corresponds to the case where  $\{A_{\ell}\}_{\ell=1}^N$ are
not jointly diagonalizable,
and the second row to the case where $\{A_{\ell}\}_{\ell=1}^N$ are
approximately jointly diagonalizable.
Recall that the FG-algorithm is only applicable to real symmetric matrices,
so its results are omitted (marked as ``--'') in the complex case.
We summarize our observation from the table as follows:
\begin{enumerate}[1)]

  \item The FG-algorithm shows far less competitive performance: much longer CPU times,
	  larger $\varepsilon_{\KKT}$, and smaller objective values, compared to the other approaches.

  \item
	  Jacobi's method performs surprisingly well for the approximately jointly diagonalizable case, 
	  outperforming NPDo in three out of four times
	  -- due to its eventual quadratic convergence~\cite{bubm:1993}.

  \item NPDo clearly outperforms Jacobi's method for the not jointly diagonalizable case,
	  except in the real case with $n=100$.

  \item
	  The not jointly diagonalizable case is considerably more challenging
      for all three approaches compared to the approximately jointly
	  diagonalizable case.
	  For example, in the real case with $n=200$,
      NPDo takes $2.67$ seconds for approximately jointly
	  diagonalizable $\{A_{\ell}\}_{\ell=1}^N$, but $8.29\cdot 10^2$ seconds for
      not jointly diagonalizable case -- a time ratio of $31.0$.
	  Time ratios are even higher for Jacobi's method and the FG-algorithm.

\end{enumerate}

Finally, to visually demonstrate that the mass of each $A_{\ell}$
becomes concentrated on its diagonal after transformation,
we plot in Figure~\ref{fig:fullJD-heat-A1} some sample heatmaps of the
$10$-by-$10$ principal submatrix corresponding to the $10$ largest diagonal entries
of the converged $P^{\HH}A_1P$, as computed by NPDo,
for the set of real test matrices of size 100-by-100 as used in~\Cref{tbl:fullJD}.
The use of $10$-by-$10$ is because otherwise little information can
	be read out from the heatmap of the whole matrix.
It is clearly visible that
for the not jointly diagonalizable case, the diagonal entries are substantially larger than
the off-diagonal ones,
while in the approximately jointly diagonalizable case, the off-diagonal entries are nearly negligible.

\subsection{Experiments on \pjd}\label{ssec:egs-pjd}
We now test on larger matrices and consider the \pjd\ problem for the approximately jointly diagonalizable case to various degree,
where a partial diagonalizer $P\in\STM{n}{k}$ is sought.
We apply both
the NPDo (i.e.,~\Cref{alg:master-NPDo:g} combined with~\Cref{alg:master-NPDo})
and accNPDo (i.e.,~\Cref{alg:master-NPDo:g} combined with~\Cref{alg:JBDvLOCG})
to compute $P$.
For comparison, we also include results obtained by computing a full
diagonalizer and then selecting the top $k$ diagonal directions, as
previously described in~\Cref{rk:fjd}.
This scheme can be summarized as follows:
\begin{enumerate}[(1)]
	\item
		Compute a full diagonalizer $P\equiv
		[\bp_1,\bp_2\ldots,\bp_n]\in\STM{n}{n}$
		(e.g., by Jacobi's method),
		along with transformed matrices
   		$\what A_{\ell}=P^{\HH}A_{\ell}P\equiv [a_{\ell;ij}]_{1\leq
		i,j\leq n}$ for
		$\ell=1,\dots, N$.
	\item Compute $d_i=\sum_{\ell=1}^N (a_{\ell;ii})^2$ for $i=1,\dots, n$ and
		sort $\{d_i\}_{i=1}^n$ in descending order
		$\{d_{j_i}\}_{i=1}^n$, where $(j_1,j_2,\ldots,j_n)$ is a
		permutation of $(1,2,\ldots,n)$.
  \item Set $P=[\bp_{j_1},\ldots,\bp_{j_k}]\in\STM{k}{n}$,
        with the corresponding objective value from~\eqref{eq:opt-pjbd}
		given by $f(P)=\sum_{i=1}^kd_{j_i}$.
\end{enumerate}
This method is at the cost of a full \jd,
unlike NPDo and accNPDo, which operate directly on the partial \jd\ problem.
In our experiment, we use Jacobi's method to obtain the full JD.
The FG-algorithm is excluded from this experiment
due to its less competitive performance as previously observed
and the need to work with general
(not necessarily positive definite) Hermitian matrices.

For the test matrices, each $A_{\ell}$ is generated using \eqref{eq:aJD-2b} with
\eqref{eq:NJD-B}, \eqref{eq:aJD-1}, and
\begin{equation}\label{eq:aJD-3'}
D_{\ell}={\tt diag}(10\cdot{\tt randn}(n,1)),
\end{equation}
and $\eta$ ranging from $10^{-3}$ to $10^0$.
The resulting $A_{\ell}$ are generally Hermitian indefinite, and thus
\Cref{alg:master-NPDo:g} requires the shifting steps as preprocessing.
The matrices $\{A_{\ell}\}_{\ell=1}^N$ may be regarded as
approximately jointly diagonalizable to a certain degree for $\eta\le 10^{-1}$,
but not jointly diagonalizable when $\eta\approx 10^0$ or larger.
In our experiments, we fix $N=10$, $k=10$, and vary $n$ from $10^2$ to $10^3$.
We have done rather extensive tests using real and complex matrices
and observed similar algorithmic behaviors in both cases.
Therefore, we report numerical results only for the more general
case where $A_{\ell}$ are complex and Hermitian.

\begin{figure}[t]
{\centering
\begin{tabular}{lcccc}
\rotatebox{90}{\hspace*{1.2cm}$\eta=10^0$} &
\resizebox*{0.22\textwidth}{0.14\textheight}{\includegraphics{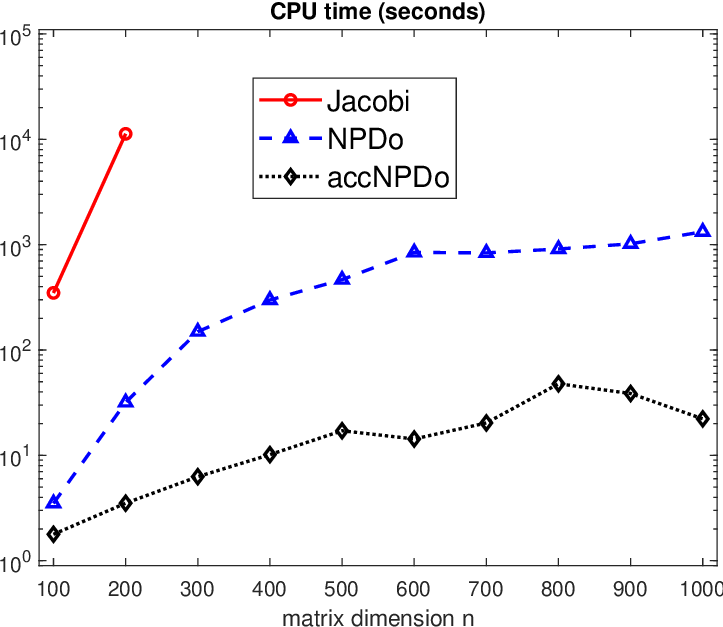}}
  & \resizebox*{0.22\textwidth}{0.14\textheight}{\includegraphics{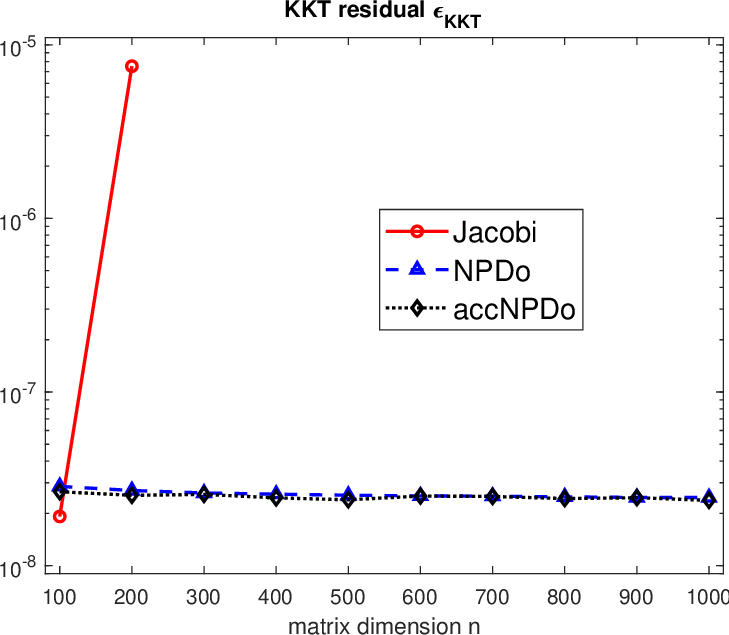}}
  & \resizebox*{0.22\textwidth}{0.14\textheight}{\includegraphics{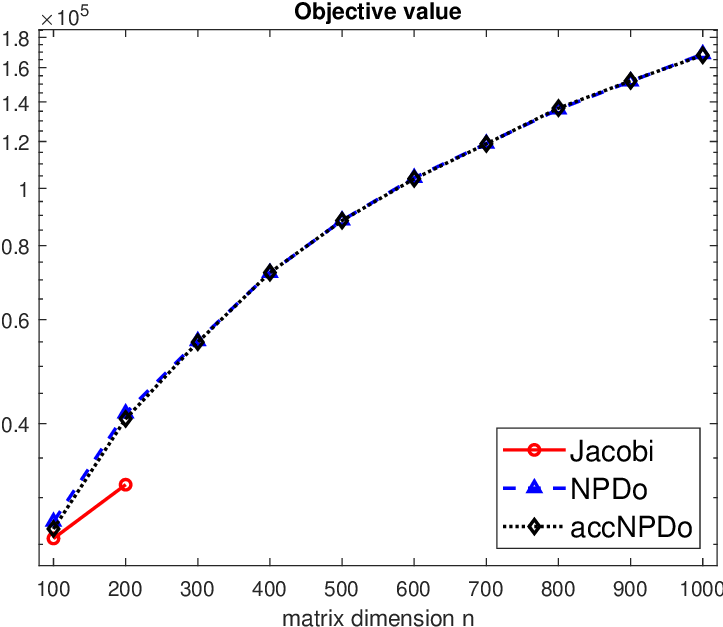}}
  & \resizebox*{0.22\textwidth}{0.14\textheight}{\includegraphics{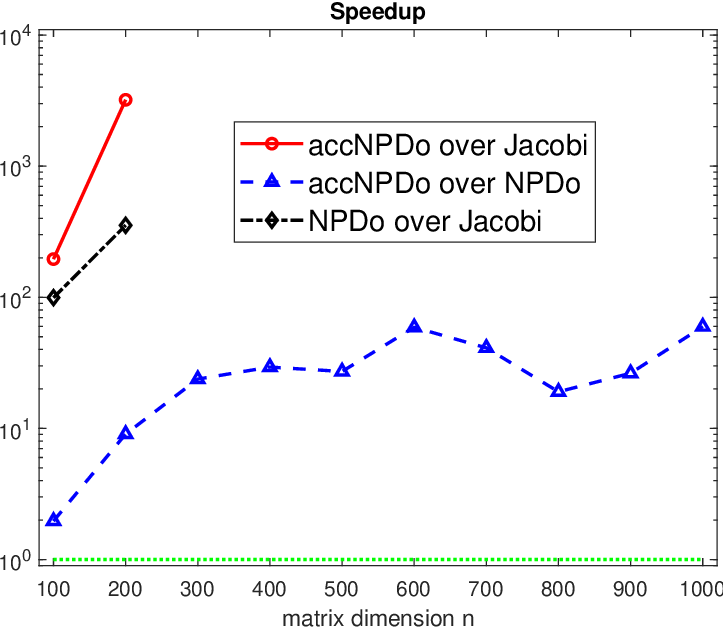}} \\
\rotatebox{90}{\hspace*{1.2cm}$\eta=10^{-1}$} &
\resizebox*{0.22\textwidth}{0.14\textheight}{\includegraphics{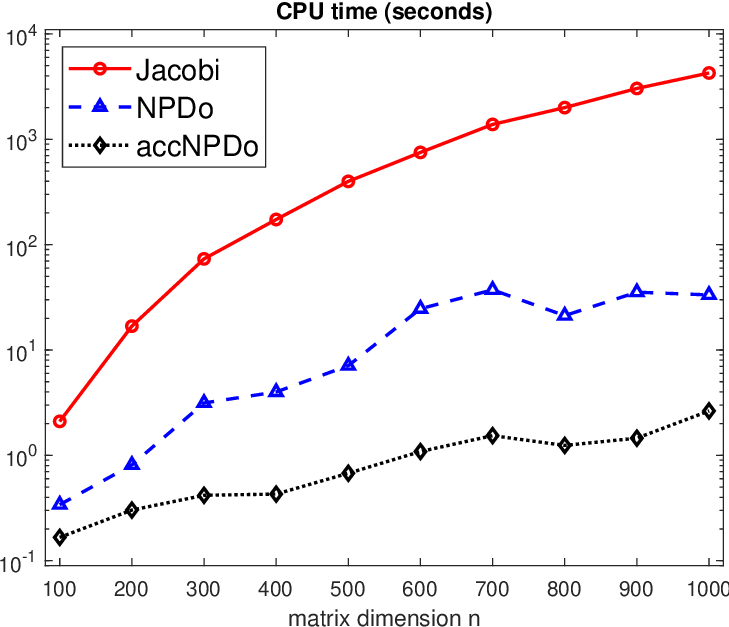}}
  & \resizebox*{0.22\textwidth}{0.14\textheight}{\includegraphics{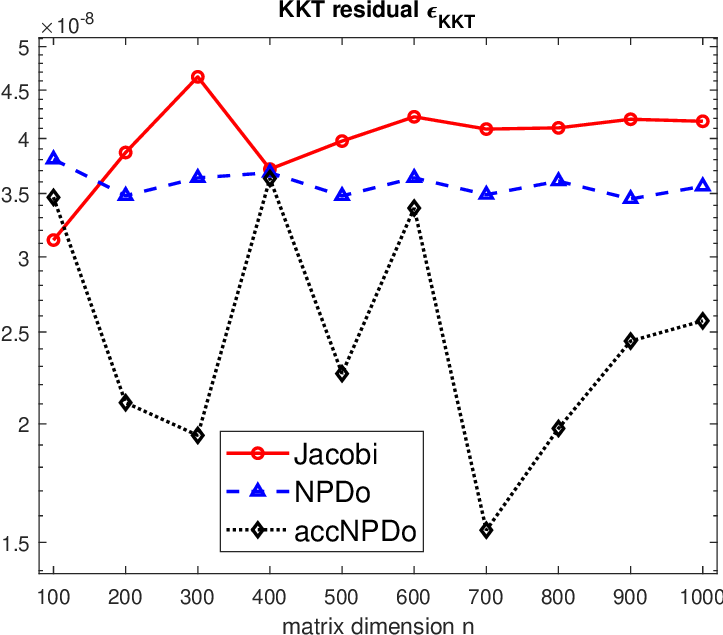}}
  & \resizebox*{0.22\textwidth}{0.14\textheight}{\includegraphics{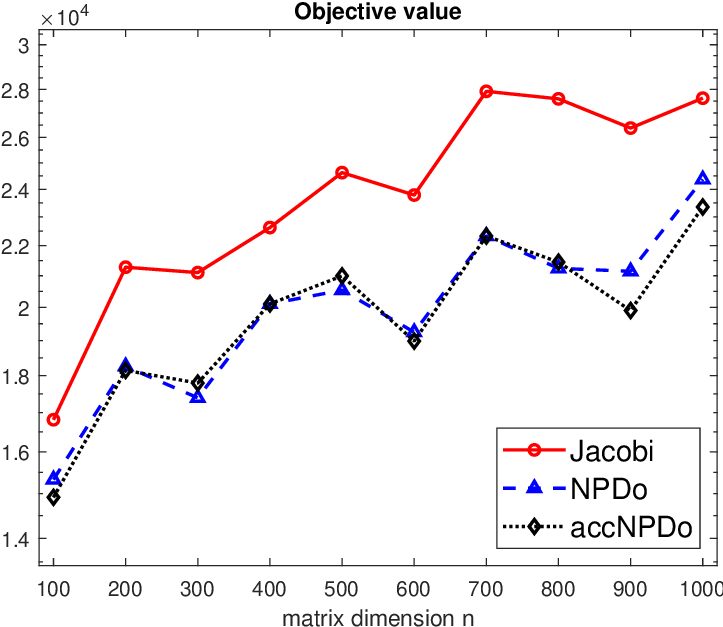}}
  & \resizebox*{0.22\textwidth}{0.14\textheight}{\includegraphics{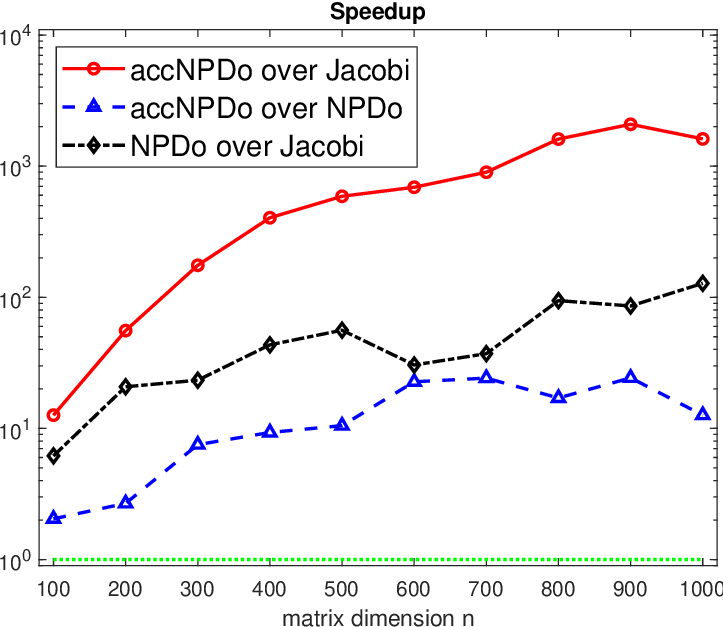}} \\
\rotatebox{90}{\hspace*{1.2cm}$\eta=10^{-2}$} &
\resizebox*{0.22\textwidth}{0.14\textheight}{\includegraphics{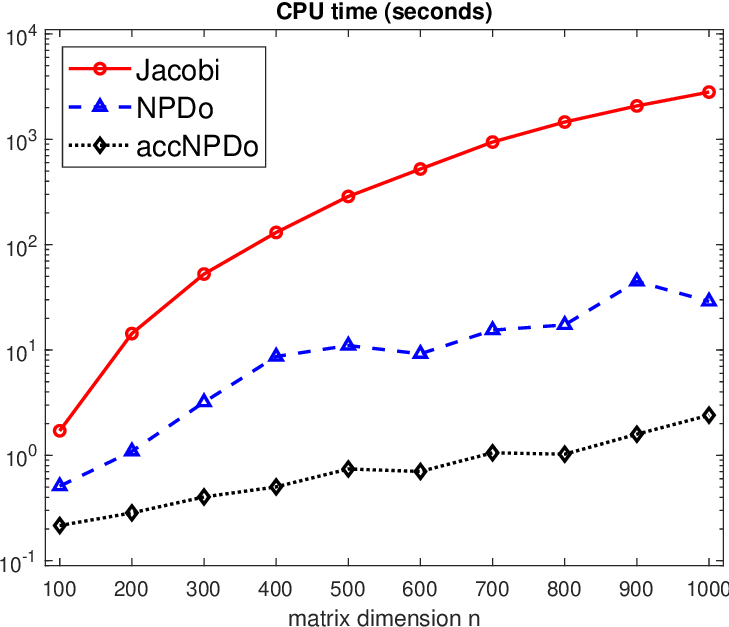}}
  & \resizebox*{0.22\textwidth}{0.14\textheight}{\includegraphics{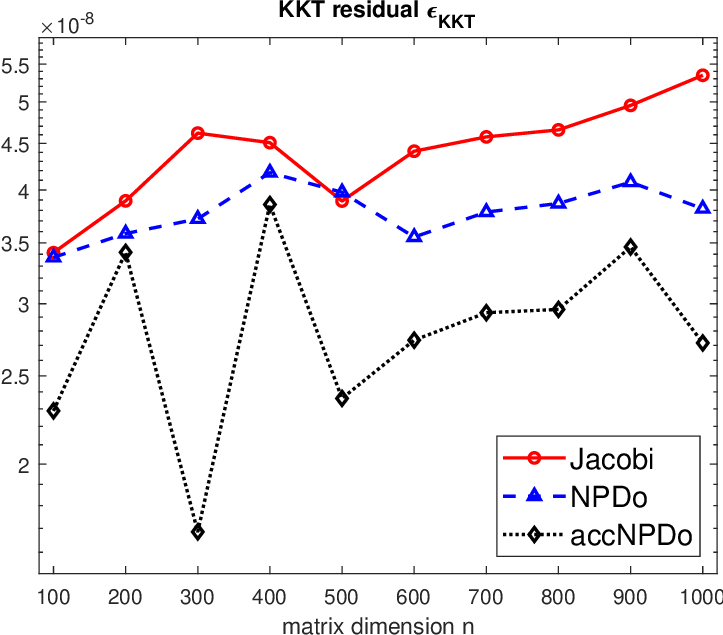}}
  & \resizebox*{0.22\textwidth}{0.14\textheight}{\includegraphics{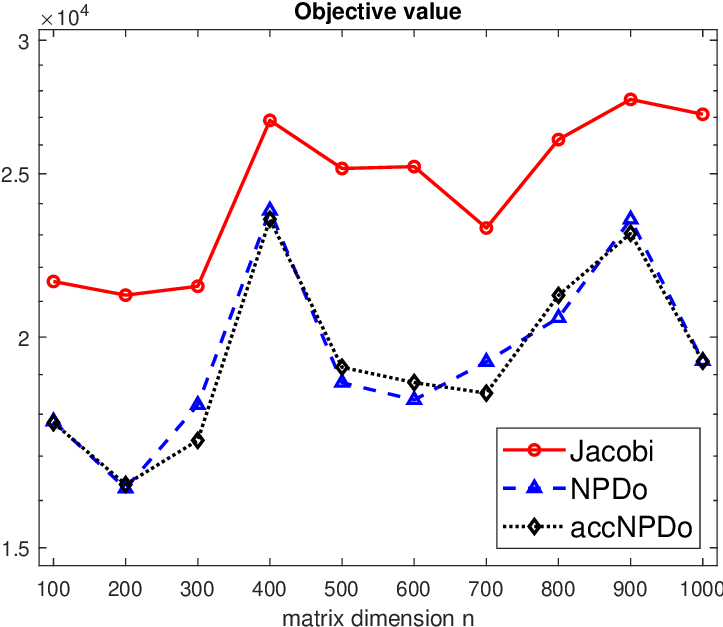}}
  & \resizebox*{0.22\textwidth}{0.14\textheight}{\includegraphics{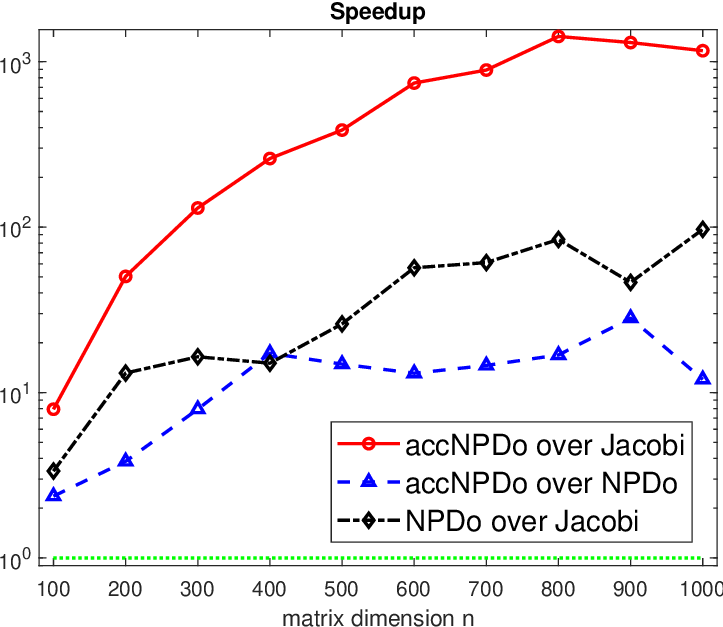}}  \\
\rotatebox{90}{\hspace*{1.2cm}$\eta=10^{-3}$} &
\resizebox*{0.22\textwidth}{0.14\textheight}{\includegraphics{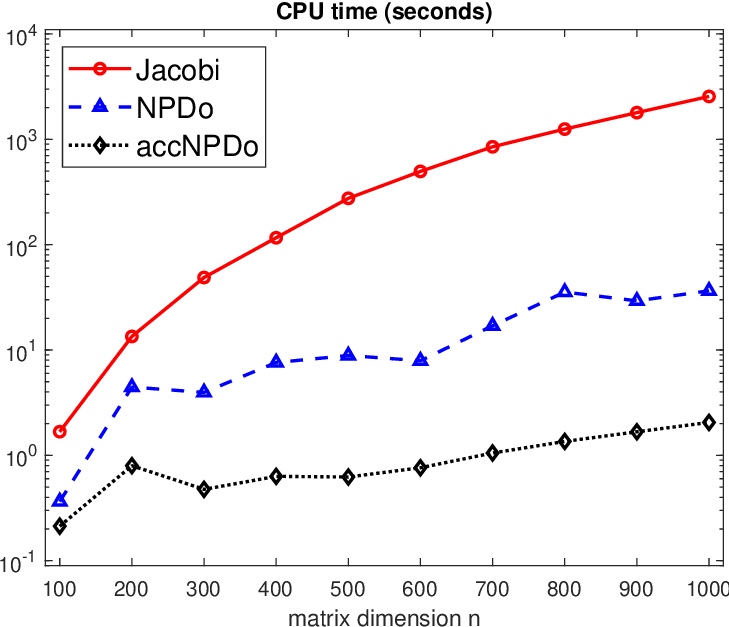}}
  & \resizebox*{0.22\textwidth}{0.14\textheight}{\includegraphics{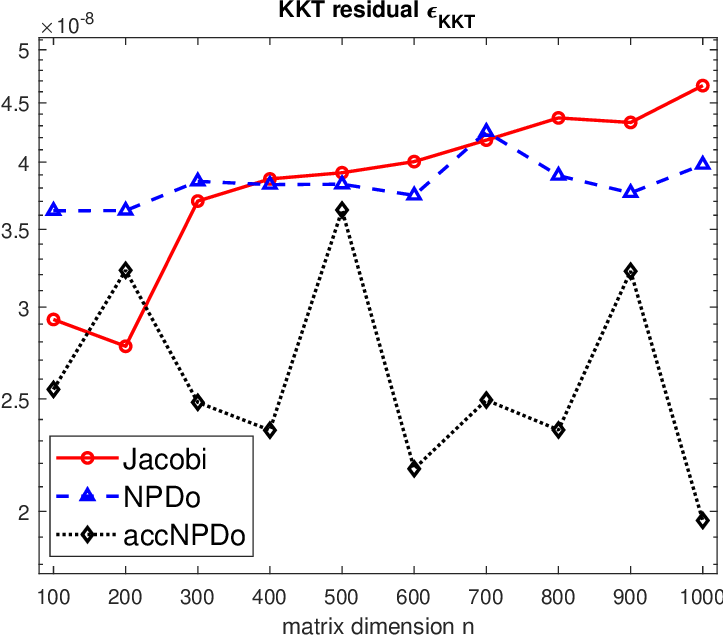}}
  & \resizebox*{0.22\textwidth}{0.14\textheight}{\includegraphics{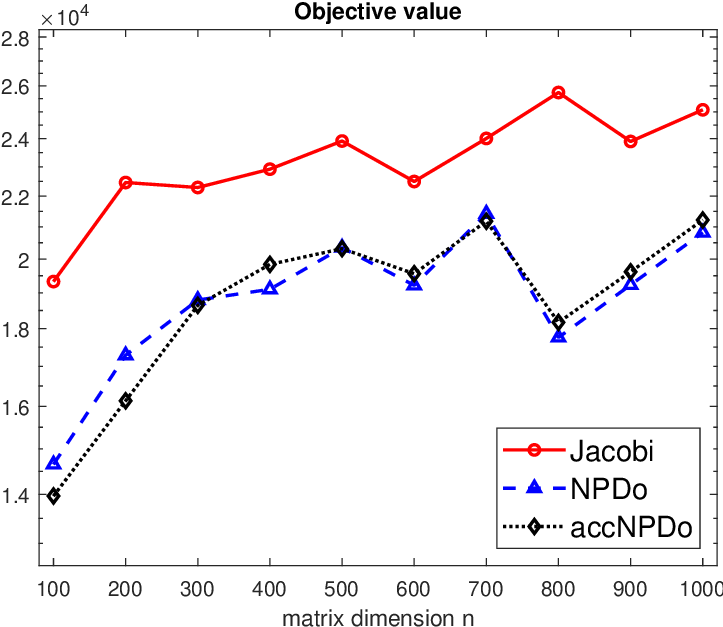}}
  & \resizebox*{0.22\textwidth}{0.14\textheight}{\includegraphics{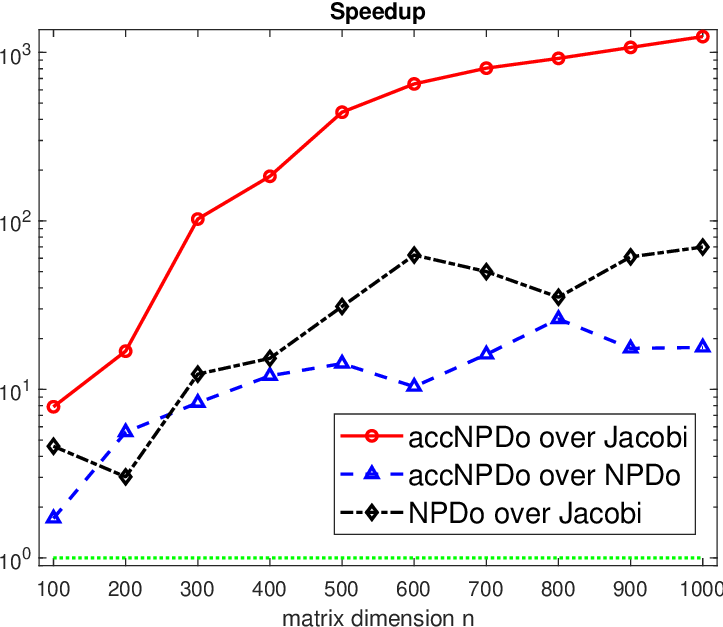}}
\end{tabular}\par
}
\vspace{-0.15 cm}
\caption{\small
   Performance  by Jacobi's method, NPDo, and accNPDo for  \pjd\ of complex
   and Hermitian $\{A_{\ell}\}_{\ell=1}^N$ generated according to
   \eqref{eq:aJD-2b} with \eqref{eq:aJD-3'}, where $k=10$, $N=10$, and $n$ varies from $10^2$ to $10^3$.
   Jacobi's method did not complete its computation within several hours at $\eta=10^0$ for $n\ge 300$.
  }
\label{fig:aJD-n-vary}
\end{figure}

\Cref{fig:aJD-n-vary} compares the performance of Jacobi's method, NPDo, and accNPDo for  \pjd\ of complex
   and Hermitian $\{A_{\ell}\}_{\ell=1}^N$.
Since $A_{\ell}$ are randomly generated, each run may produce
different results; thus
\Cref{fig:aJD-n-vary} is intended to reflect the typical behavior of the three methods.
We have the following observations.
Regarding computational time:
\begin{itemize}
  \item
	  For $\eta\le 10^{-1}$, Jacobi's method performs reasonably well,
	  but it runs significantly slower than both NPDo and accNPDo as $n$ increases.
      For example, at $\eta=10^{-3}$ and $n=10^3$, it is $1,241.5$ times slower than accNPDo and
      $69.9$ times slower than NPDo.

  \item
	  At $\eta=10^0$, Jacobi's method did not complete the computation within a few hours for $n\ge 300$,
	  due to significantly increased number of iterations (i.e., sweeps).
	  This indicates that Jacobi's method is impractical for the not-approximately jointly diagonalizable
        case, even for moderate matrix sizes.

\end{itemize}
Regarding the quality of diagonalizer:
\begin{itemize}
  \item
   	  At $\eta=10^0$, Jacobi's method fails to go beyond
	  $n=200$ within several hours as noted earlier, but NPDo and accNPDo achieves $O(10^{-8})$
	  in $\varepsilon_{\KKT}$ relatively quickly.
  \item
	  For $\eta\le 10^{-1}$, the normalized KKT residuals
	  $\varepsilon_{\KKT}$ achieved by all three methods are of
	  $O(10^{-8})$, with an edge going to accNPDo,
		  indicating comparable solution quality for the principal {\jd}
		  problem~\eqref{eq:opt-pjbd}, in terms of first-order
		  optimality.
	
  \item
	  For $\eta\le 10^{-1}$, Jacobi's method consistently achieves the largest objective values.
	  This is partly due to the fact that
	  its partial diagonalizer is extracted from a full \jd.
	  As noted in~\Cref{rk:fjd},  when the matrices are exactly jointly diagonalizable,
	  such a partial diagonalizer corresponds to the global optimal solution of
	  the  \pjd\ problem~\eqref{eq:opt-pjbd}.
	  Therefore, for matrices that are approximately jointly diagonalizable,
	  Jacobi's method has a good chance to find an
	  better global maximizer of the objective function,
	  while NPDo and accNPDo may converge to local maximums.
	  This explains  the larger objective values by Jacobi's method in
	  the experiments.
\end{itemize}
Overall, accNPDo is the fastest among the three and, on average, yields
the smallest normalized KKT residuals $\varepsilon_{\KKT}$.

\Cref{tbl:aJD-Jac} reports the number of sweeps (outer iterations)
required by Jacobi's method for $\eta\le 10^{-1}$.
Notably, the number of sweeps remains relatively small and
grows slowly with $n$, especially for smaller values of $\eta$.
However, since each sweep requires $O(Nn^3)$ flops,
the CPU time still grows substantially with $n$.
Additionally, there is a clear trend of increasing number of sweeps
as the parameter $\eta$ grows,
indicating the challenges of Jacobi’s method in handling cases
where the matrices are not approximately jointly diagonalizable.

\setlength{\tabcolsep}{4pt}
\renewcommand{\arraystretch}{1.4}
\begin{table}[t]
\caption{\small Number of sweeps (outer iterations) by Jacobi's method}\label{tbl:aJD-Jac}
\centerline{
\begin{tabular}{|c||*{10}{c|}}\hline
\backslashbox{$\eta$}{$n$}
  & 100 & 200 & 300 & 400  &  500  &  600 & 700 & 800 & 900 & 1000 \\ \hline
$10^{-3}$ & 7 &   8 &   9 &  10 &  10 &  11 &  11 &  11 &  11 &  13 \\ \hline
$10^{-2}$ & 8 &   9 &  11 &  11 &  11 &  12 &  12 &  13 &  14 &  13 \\ \hline
$10^{-1}$ & 9 &  11 &  14 &  15 &  15 &  18 &  18 &  20 &  21 &  23\\ \hline
\end{tabular}
}
\end{table}

\subsection{Experiments on  \pjbd}\label{ssec:egs-JBD}
We now consider the general principal {\jbd} problem with block sizes $k_i > 1$.
For this task, only NPDo and accNPDo are applicable.
The test matrices $\{A_{\ell}\}_{\ell=1}^N$ are generated according to \eqref{eq:aJD-2b} with
\eqref{eq:NJD-B}, \eqref{eq:aJD-1}, and
with block-diagonal
\begin{equation}\label{eq:aJBD-D}
D_{\ell}=\diag(10(E_1+E_1^{\HH}),10(E_2+E_2^{\HH}),\ldots,10(E_t+E_t^{\HH})),
\end{equation}
where each $E_i = {\tt randn}(k_i)+{\tt 1i*}{\tt randn}(k_i)$.
We have also conducted extensive experiments for the real case with $E_i={\tt randn}(k_i)$ for the real case,
and observed similar algorithmic behaviors.
In the test below, we fix $N=10$, $k=10$, and $k_i=2$, for $i=1,\dots,5$
(i.e., $t=5$ blocks of equal size 2-by-2),
and vary $n$ from $10^2$ to $10^3$, and $\eta$ from $10^0$ down to $10^{-3}$.

\begin{figure}[t]
{\centering
\begin{tabular}{lcccc}
\rotatebox{90}{\hspace*{1.2cm}$\eta=10^0$} &
\resizebox*{0.22\textwidth}{0.14\textheight}{\includegraphics{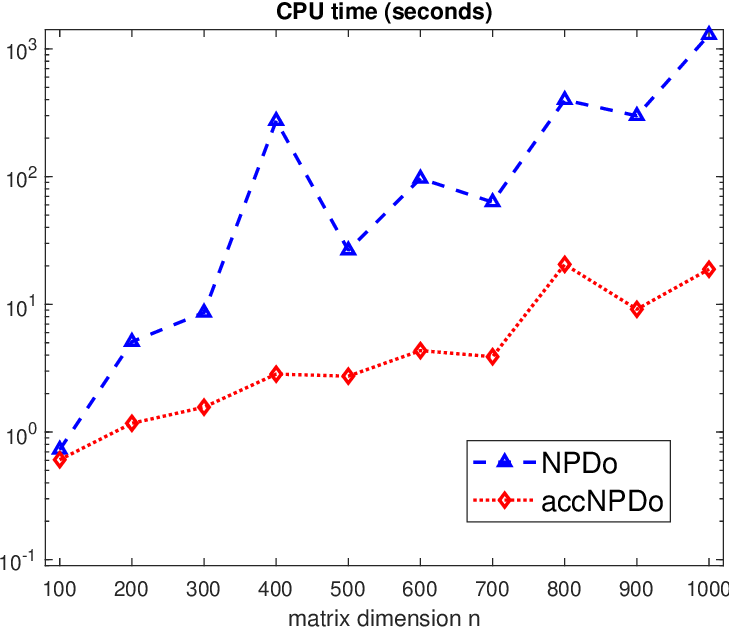}}
  & \resizebox*{0.22\textwidth}{0.14\textheight}{\includegraphics{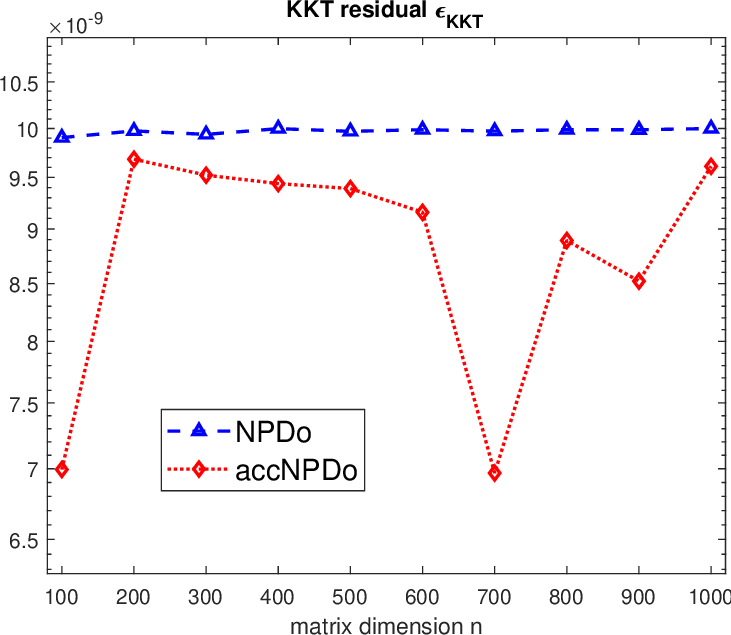}}
  & \resizebox*{0.22\textwidth}{0.14\textheight}{\includegraphics{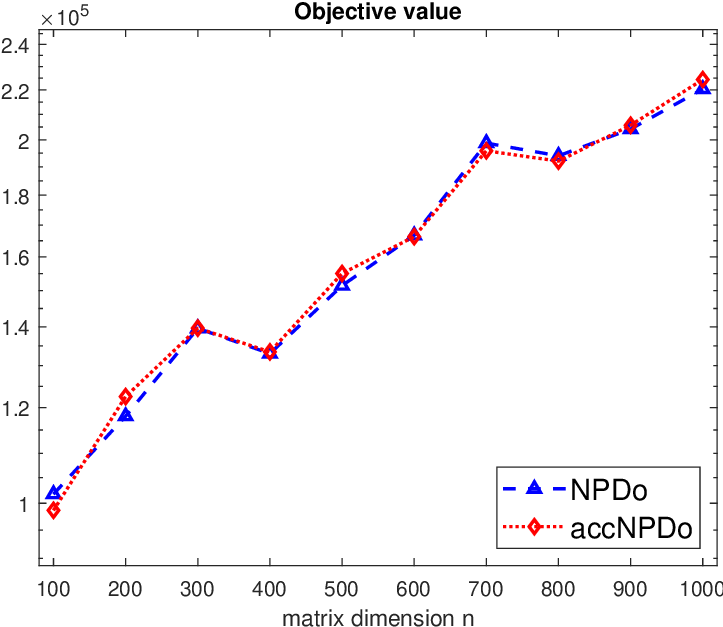}}
  & \resizebox*{0.22\textwidth}{0.14\textheight}{\includegraphics{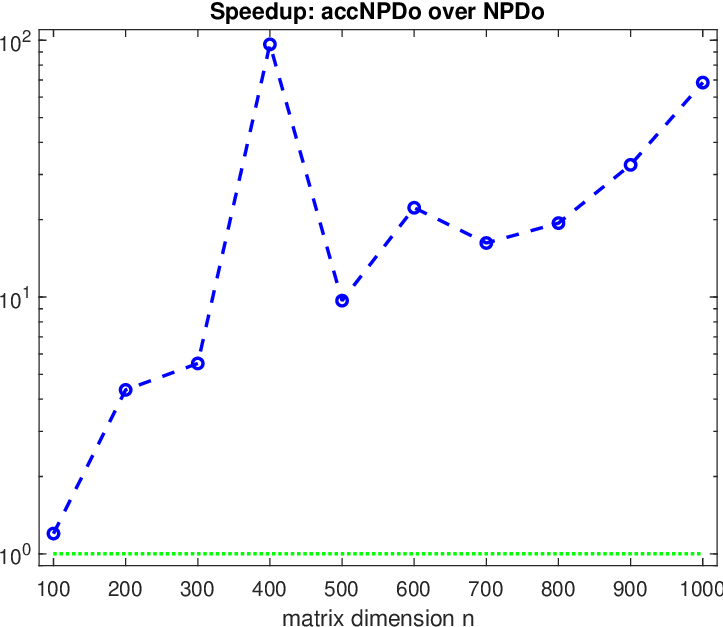}} \\
\rotatebox{90}{\hspace*{1.2cm}$\eta=10^{-1}$} &
\resizebox*{0.22\textwidth}{0.14\textheight}{\includegraphics{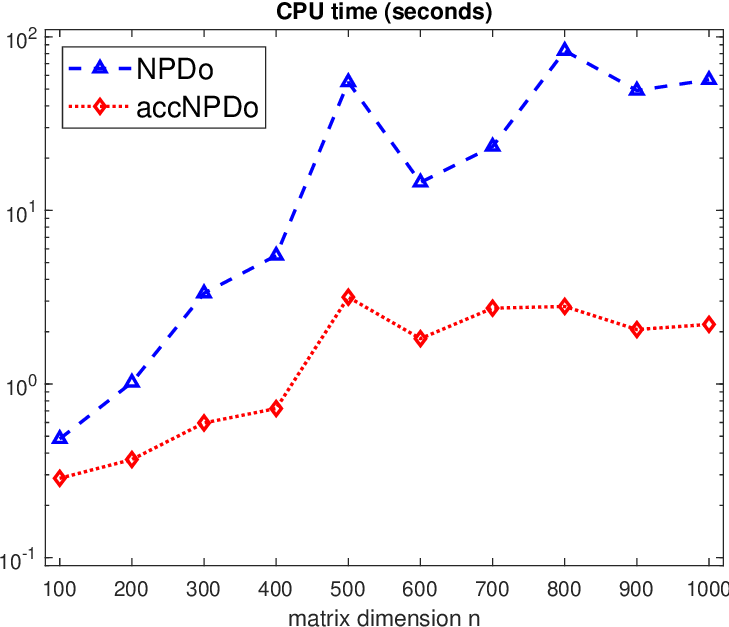}}
  & \resizebox*{0.22\textwidth}{0.14\textheight}{\includegraphics{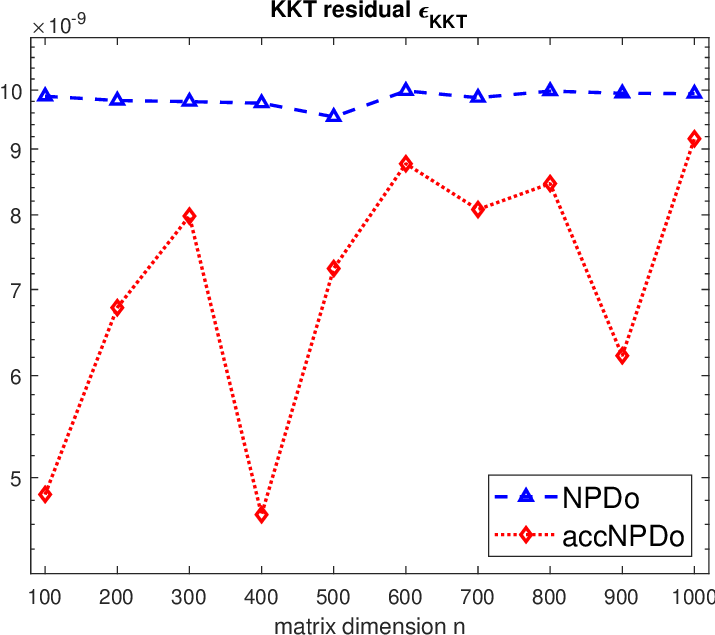}}
  & \resizebox*{0.22\textwidth}{0.14\textheight}{\includegraphics{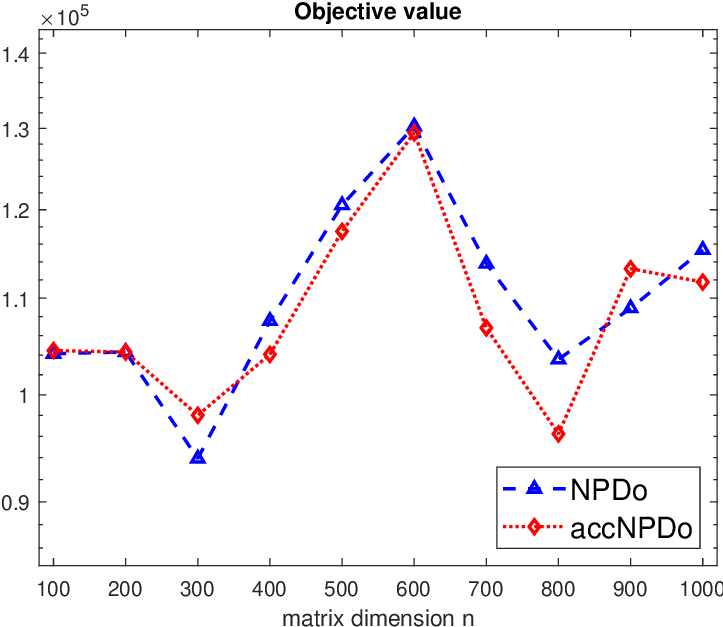}}
  & \resizebox*{0.22\textwidth}{0.14\textheight}{\includegraphics{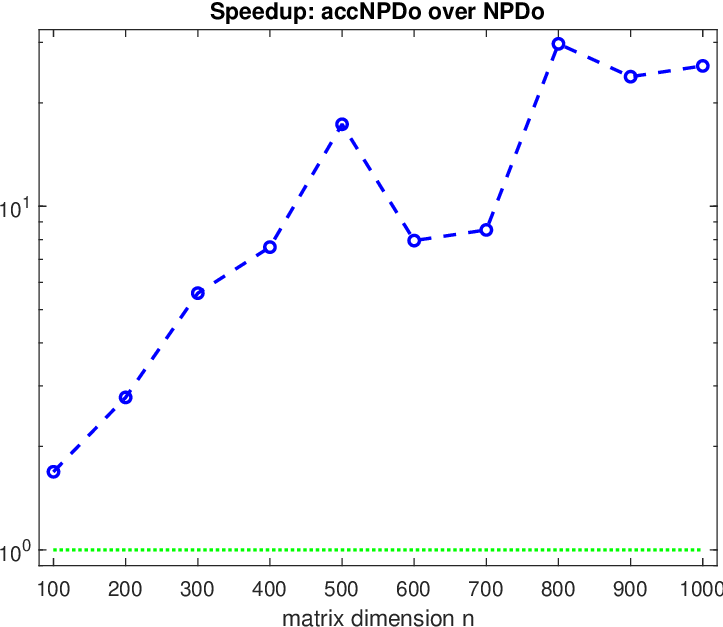}} \\
\rotatebox{90}{\hspace*{1.2cm}$\eta=10^{-2}$} &
\resizebox*{0.22\textwidth}{0.14\textheight}{\includegraphics{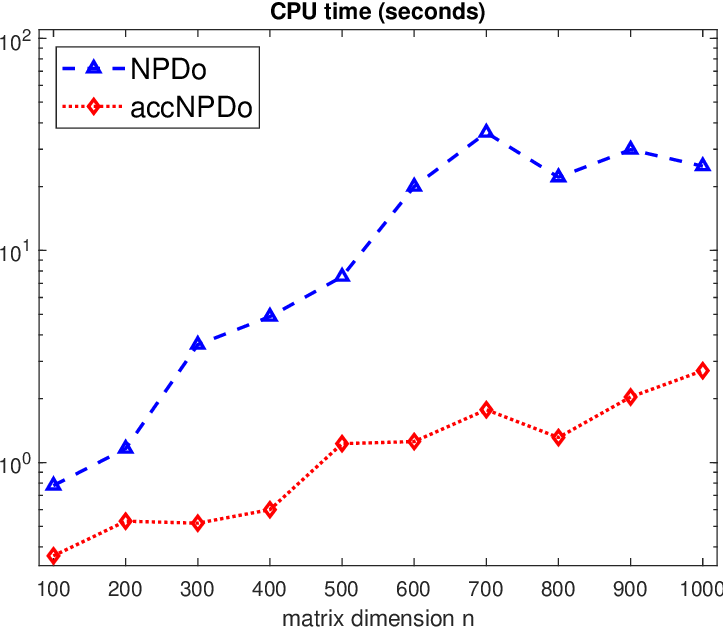}}
  & \resizebox*{0.22\textwidth}{0.14\textheight}{\includegraphics{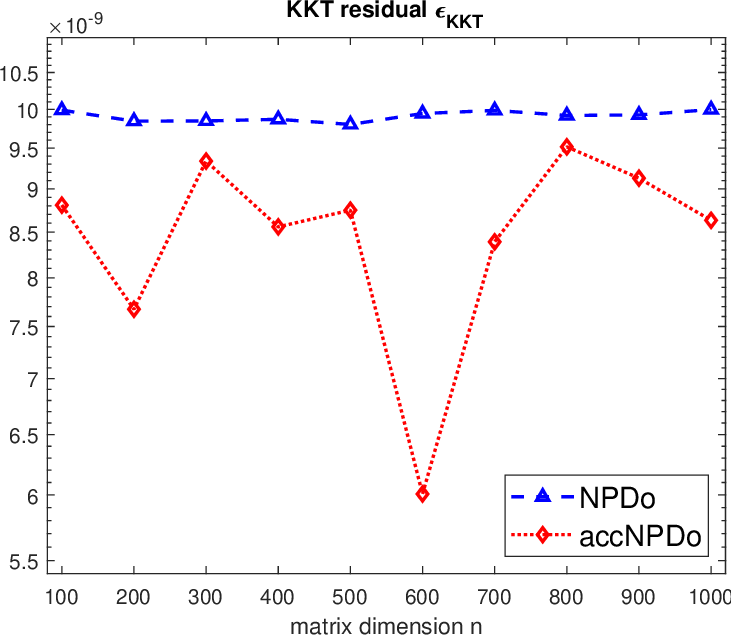}}
  & \resizebox*{0.22\textwidth}{0.14\textheight}{\includegraphics{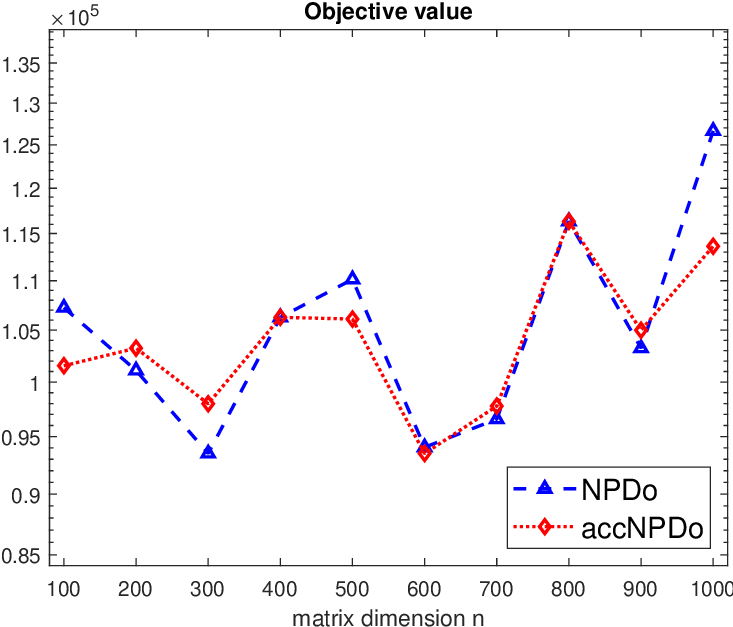}}
  & \resizebox*{0.22\textwidth}{0.14\textheight}{\includegraphics{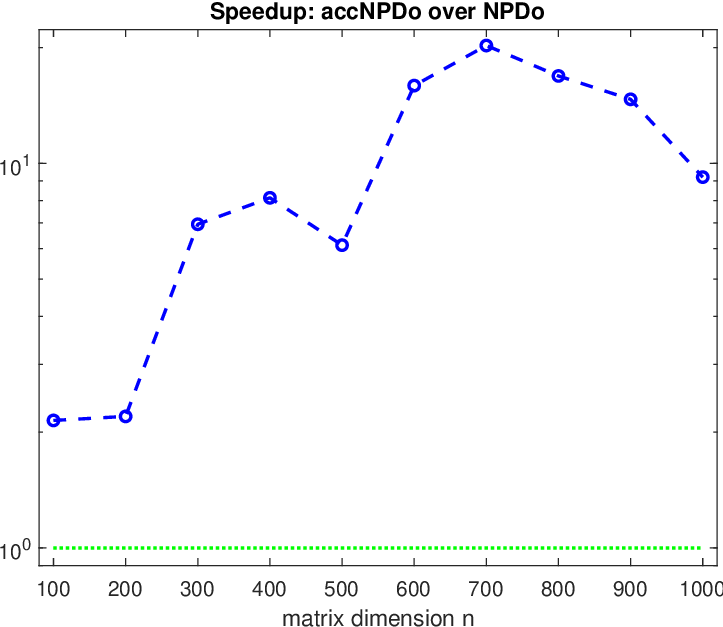}}  \\
\rotatebox{90}{\hspace*{1.2cm}$\eta=10^{-3}$} &
\resizebox*{0.22\textwidth}{0.14\textheight}{\includegraphics{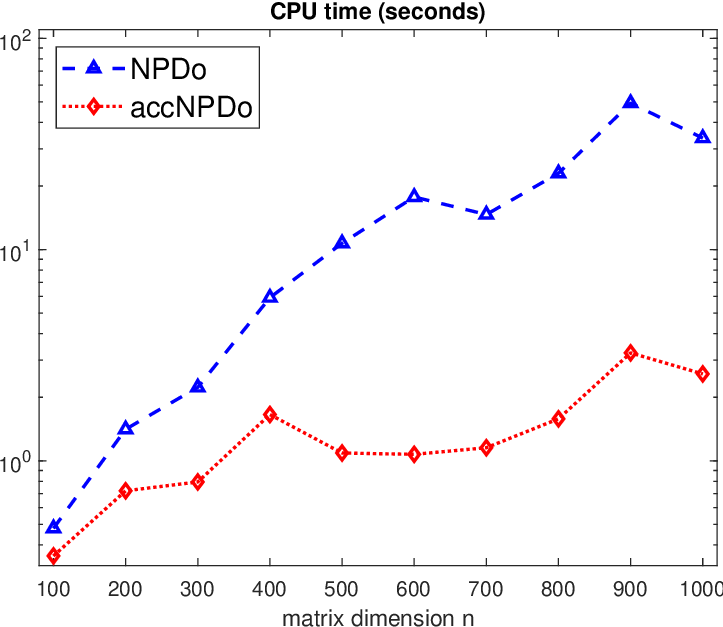}}
  & \resizebox*{0.22\textwidth}{0.14\textheight}{\includegraphics{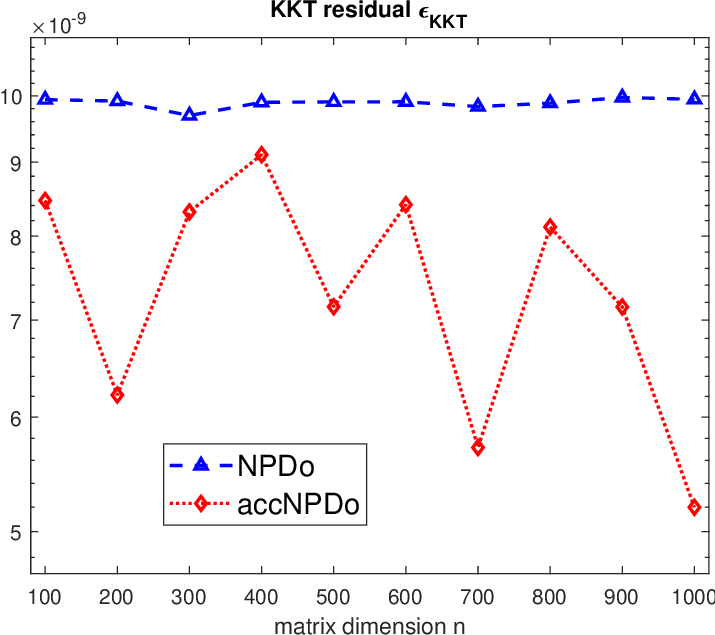}}
  & \resizebox*{0.22\textwidth}{0.14\textheight}{\includegraphics{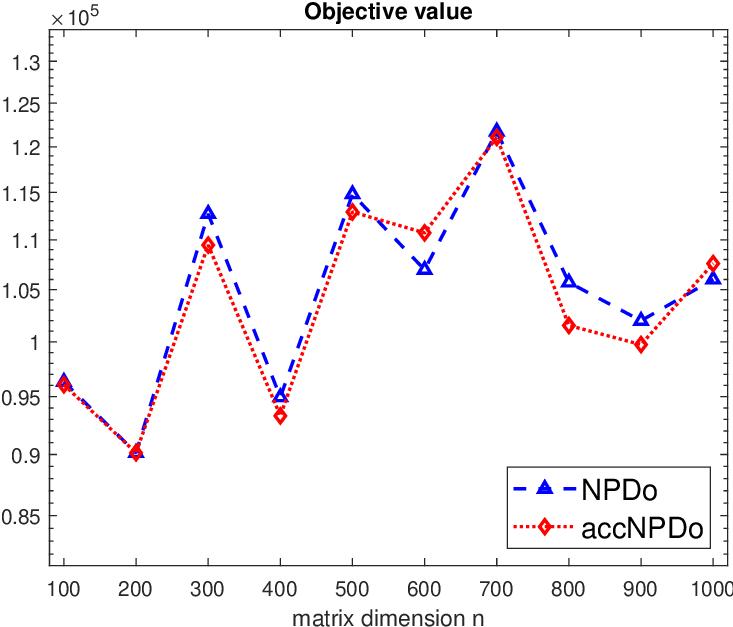}}
  & \resizebox*{0.22\textwidth}{0.14\textheight}{\includegraphics{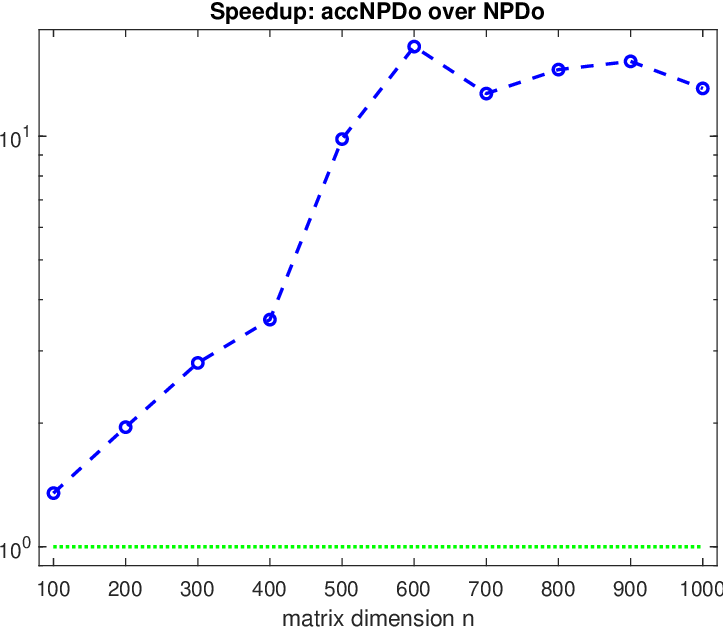}}
\end{tabular}\par
}
\vspace{-0.15 cm}
\caption{\small
   Performance by  NPDo and accNPDo for \pjbd\  of complex and Hermitian
   $\{A_{\ell}\}_{\ell=1}^N$ generated according to
   \eqref{eq:aJD-2b} with \eqref{eq:aJBD-D},
   where $k=10$, $N=10$, and $n$ varies from $10^2$ to $10^3$.
  }
\label{fig:aJBD-n-vary}
\end{figure}

\Cref{fig:aJBD-n-vary} compares the performance of NPDo and accNPDo.
We make the following observations:
(a) The normalized KKT residuals $\varepsilon_{\KKT}$ achieved by the two methods
are of $O(10^{-8})$, with an edge going to accNPDo,
indicating similar solution quality for the computed  \pjbd;
(b) Both methods achieves comparable optimal objective values;
(c) accNPDo is consistently faster.

\subsection{Effect of initial guess}
%
We have been using  random initial $P^{(0)}$
as generated in \eqref{eq:P0-init} for use by NPDo and accNPDo.
With the same initial for each generated testing set $\{A_{\ell}\}_{\ell=1}^N$, both NPDo and accNPDo have produced decent stationary points, as they should according to the NPDo theory
\cite{li:2024}. Unfortunately, the theory does not guarantee that these stationary points are global maximizers.
In fact, often they are not.
This is clear from the numerical results in
subsection~\ref{ssec:egs-pjd} where we observed that Jacobi's method produced the largest objective value for each
testing set generated with $\eta\le 10^{-1}$ (the approximately jointly diagonalizable case), while all three methods
ended up with the normalized KKT residuals $\varepsilon_{\KKT}=O(10^{-8})$, indicating the computed maximizers
have similar quality as stationary points.
An intuitive, but not conclusive, explanation is made in \Cref{rk:fjd} to argue that the \pjd\ maximizer induced from the solution by Jacobi's method is more likely a global maximizer for the approximately jointly diagonalizable case,
explaining why Jacobi's method produced the largest objective value.
From a different perspective, that explanation  suggests an initial guess strategy, i.e., the eigen-basis matrix of
$A_{\tot}$ as in~\eqref{eq:Atot} associated with its $k$ largest eigenvalues could serve as an initial $P^{(0)}$.

We have experimented with this strategy.
While generating the numerical results with random initials that yield \Cref{fig:aJD-n-vary}, at the same time we also run NPDo and accNPDo with an initial computed by LOBPCG (without preconditioning) on $A_{\tot}$ as we just mentioned.
Note that we compute a very rough approximate eigen-basis matrix $P^{(0)}$ associated with the $k$ largest eigenvalues of
$A_{\tot}$ so that
$$
\frac {\big\|A_{\tot}P^{(0)}-P^{(0)}\big([P^{(0)}]^{\HH}A_{\tot}P^{(0)}\big)\big\|_{\F}}{\sum_{\ell=1}^N\|A_{\ell}\|_{\F}\|A_{\ell}\|_2}
   \le 10^{-3}.
$$
\Cref{fig:aJD-n-vary-init} compares the objective values by Jacobi's method, NPDo and accNPDo (with random initials),
and NPDo${}_{\init}$ and accNPDo${}_{\init}$ (with the new initial strategy).
\begin{figure}[t]
{\centering
\begin{tabular}{lcccc}
& $\eta=10^{0}$ & $\eta=10^{-1}$ & $\eta=10^{-2}$ & $\eta=10^{-3}$ \\ 
\rotatebox{90}{\hspace*{1.0cm}real}
  & \resizebox*{0.22\textwidth}{0.14\textheight}{\includegraphics{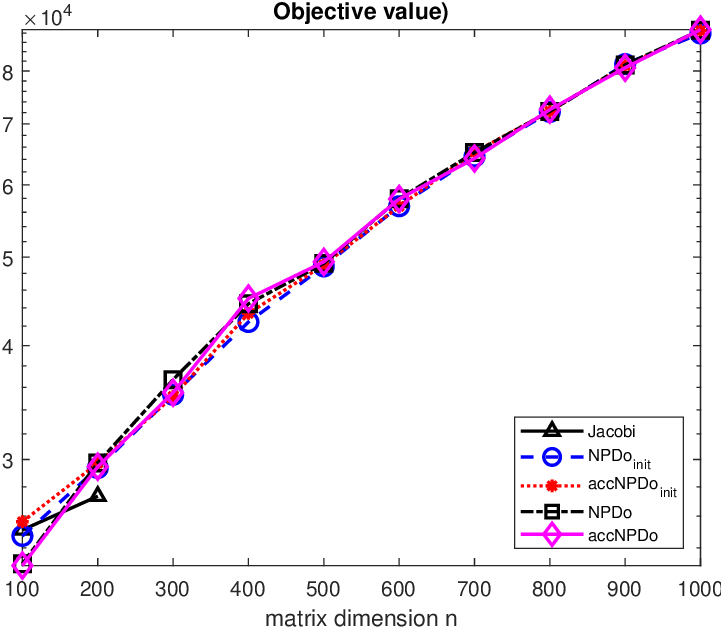}}
  & \resizebox*{0.22\textwidth}{0.14\textheight}{\includegraphics{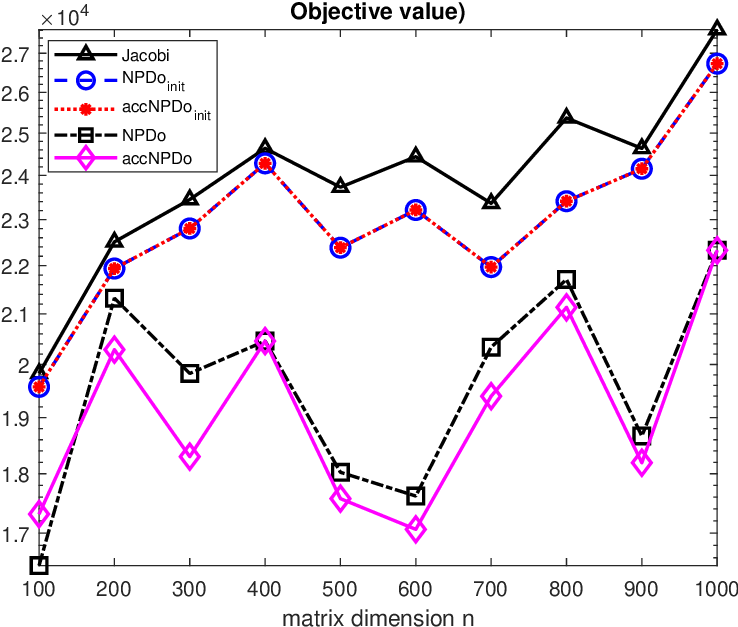}}
  & \resizebox*{0.22\textwidth}{0.14\textheight}{\includegraphics{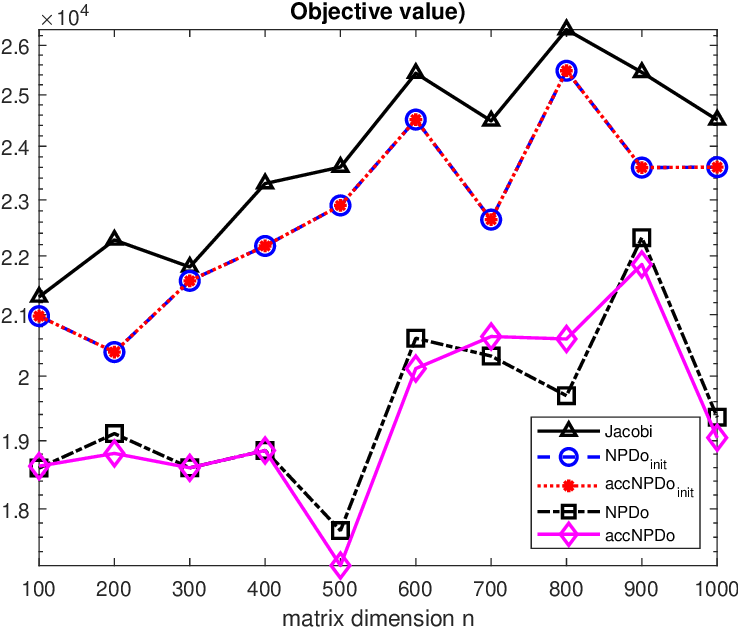}}
  & \resizebox*{0.22\textwidth}{0.14\textheight}{\includegraphics{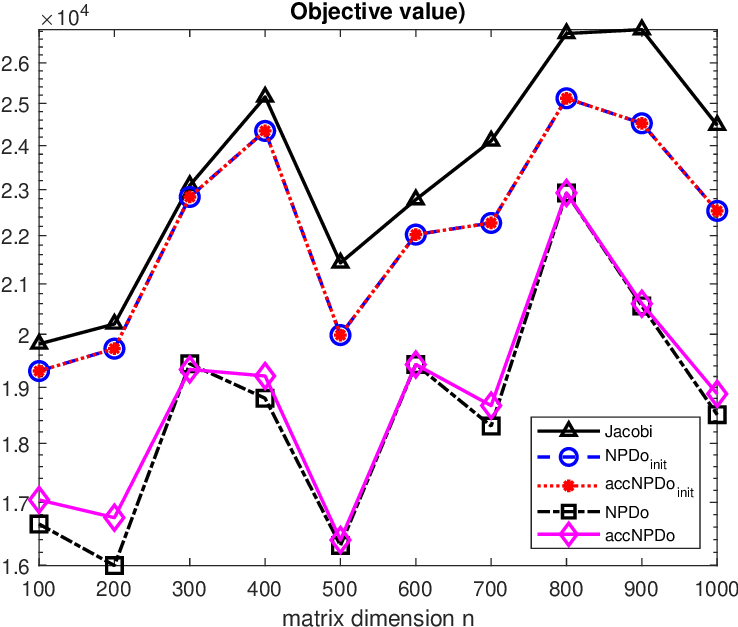}} \\ 
\rotatebox{90}{\hspace*{0.8cm}complex}
  & \resizebox*{0.22\textwidth}{0.14\textheight}{\includegraphics{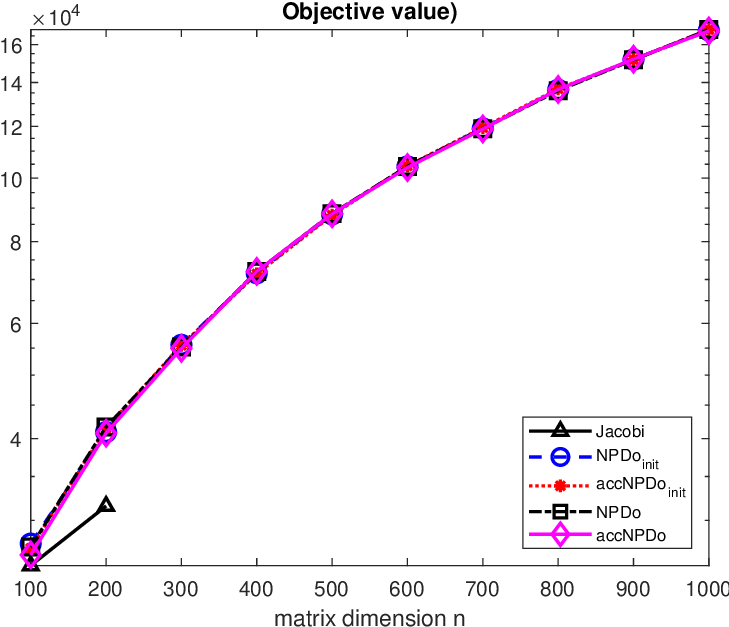}}
  & \resizebox*{0.22\textwidth}{0.14\textheight}{\includegraphics{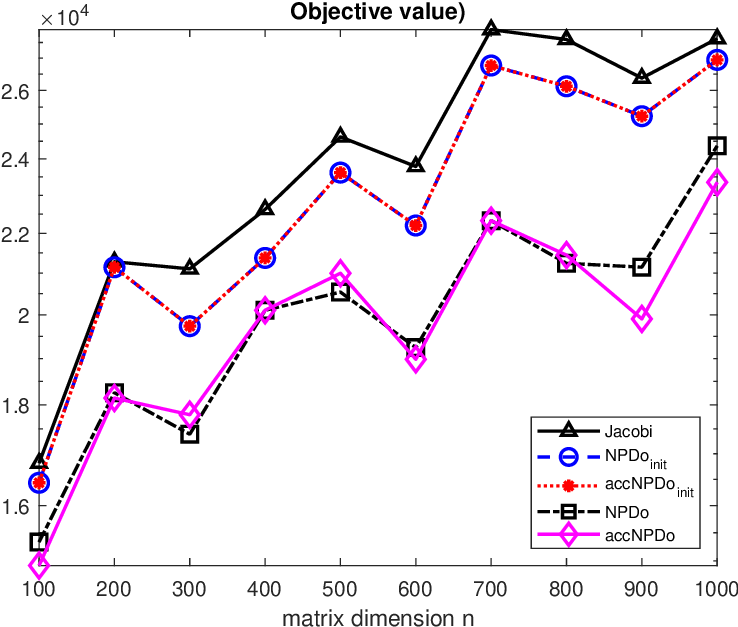}}
  & \resizebox*{0.22\textwidth}{0.14\textheight}{\includegraphics{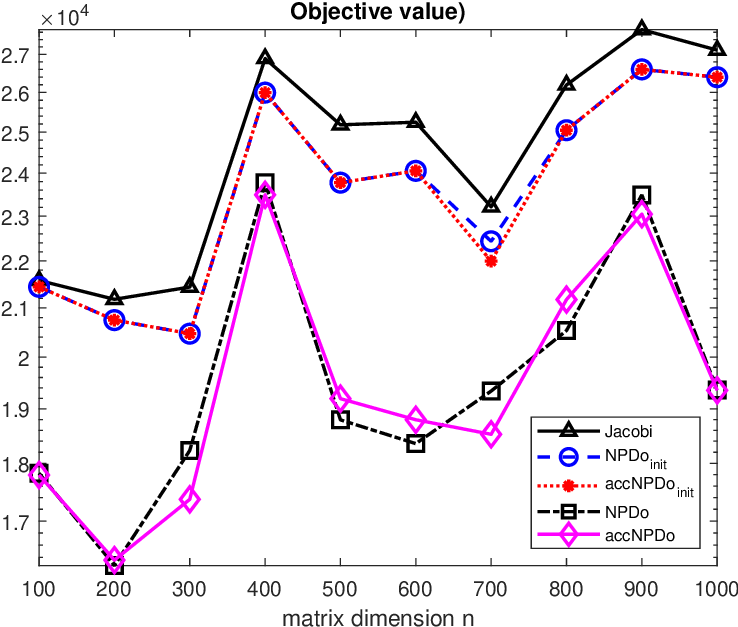}}
  & \resizebox*{0.22\textwidth}{0.14\textheight}{\includegraphics{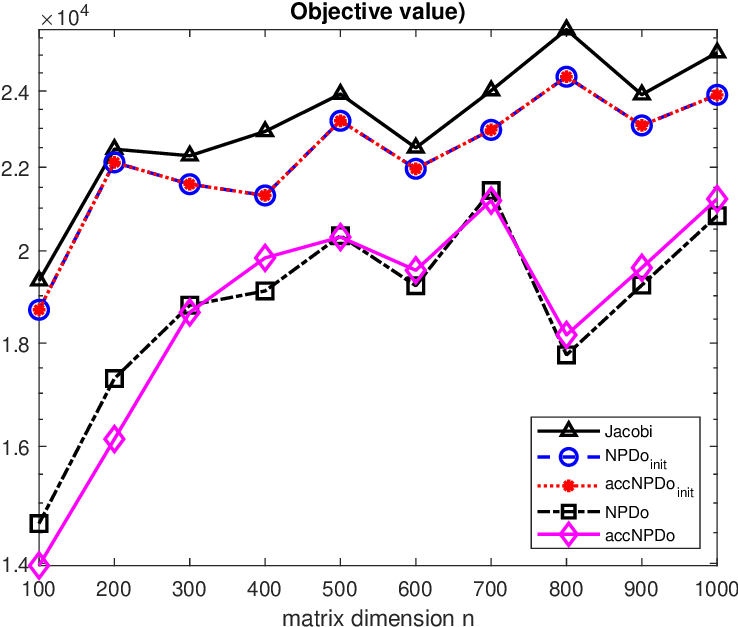}}
\end{tabular}\par
}
\vspace{-0.15 cm}
\caption{\small
   Effect of initial guess on the same testing sets that generates \Cref{fig:aJD-n-vary}.
  }
\label{fig:aJD-n-vary-init}
\end{figure}
We have the following observations:
\begin{enumerate}[(1)]
  \item Previously in \Cref{fig:aJD-n-vary}, we already noticed that Jacobi's method struggled for $\eta=10^0$
        (the not approximately jointly diagonalizable case), not only it could not finish
        its computations beyond $n=200$ within hours but also produces smaller objective values for $n\le 200$ than NPDo and accNPDo.
        \Cref{fig:aJD-n-vary-init} shows that the new initial guess strategy does not help NPDo${}_{\init}$ and accNPDo${}_{\init}$ as far as
        better stationary points are concerned.
  \item Overwhelmingly, when $\eta\le 10^{-1}$ ((the approximately jointly diagonalizable case),  the new initial strategy helps NPDo${}_{\init}$ and accNPDo${}_{\init}$ to achieve larger objective values
        than corresponding NPDo and accNPDo, respectively, that use random initials, but the improvements are still short to make
        NPDo${}_{\init}$ and accNPDo${}_{\init}$ overtake Jacobi's method.
\end{enumerate}

We also experimented with this initial guess strategy for \jbd, too, i.e., besides random initials that yield \Cref{fig:aJBD-n-vary},
at the same time we also used the initial guess strategy:  a very rough approximate eigen-basis matrix associated with the $k$ largest eigenvalues of
$A_{\tot}$. \Cref{fig:aJBD-n-vary-init} compares the objective values by NPDo and accNPDo (with random initials),
and NPDo${}_{\init}$ and accNPDo${}_{\init}$ (with the new initial strategy).
\begin{figure}[t]
{\centering
\begin{tabular}{lcccc}
& $\eta=10^{0}$ & $\eta=10^{-1}$ & $\eta=10^{-2}$ & $\eta=10^{-3}$ \\ 
\rotatebox{90}{\hspace*{1.0cm}real}
  & \resizebox*{0.22\textwidth}{0.14\textheight}{\includegraphics{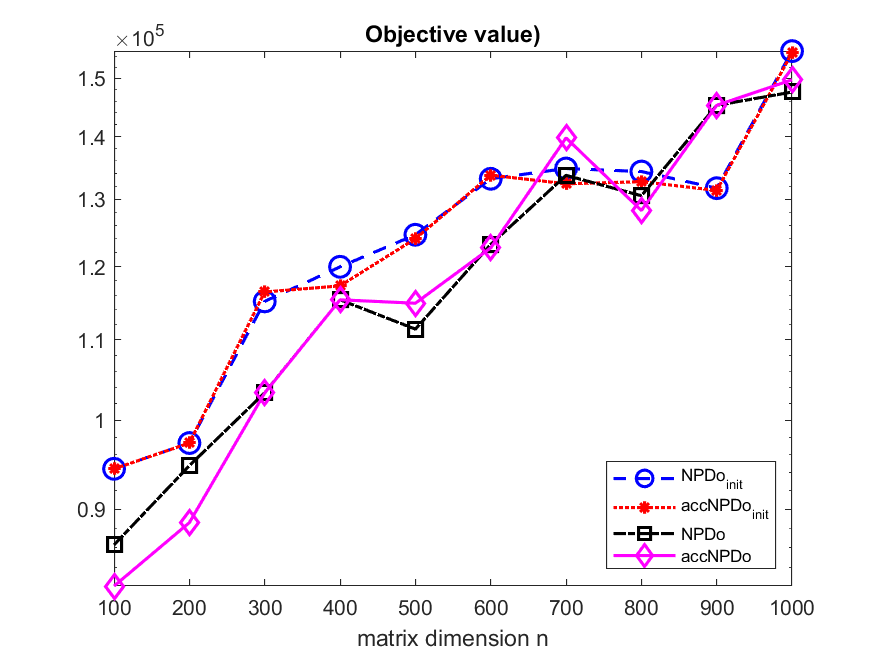}}
  & \resizebox*{0.22\textwidth}{0.14\textheight}{\includegraphics{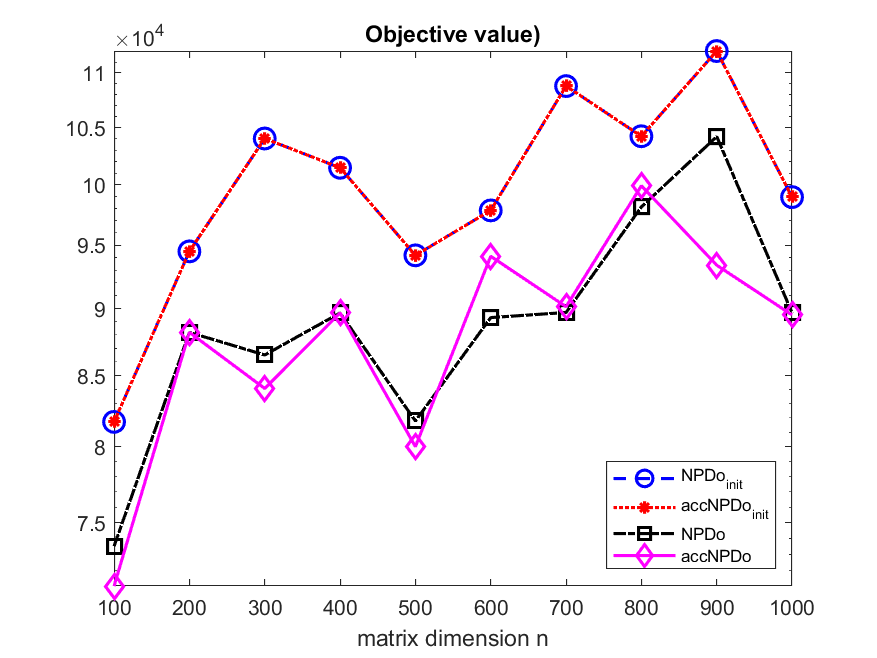}}
  & \resizebox*{0.22\textwidth}{0.14\textheight}{\includegraphics{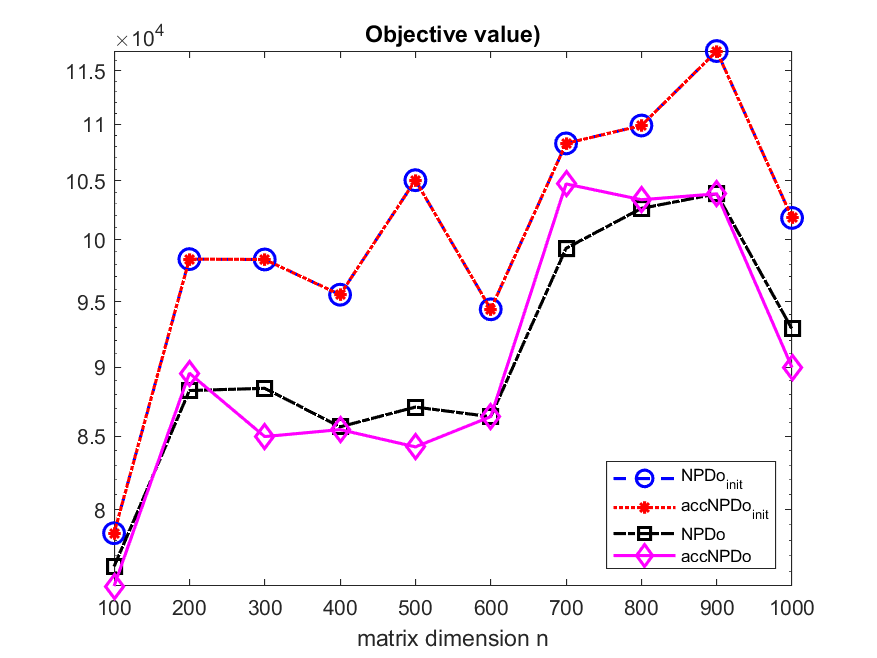}}
  & \resizebox*{0.22\textwidth}{0.14\textheight}{\includegraphics{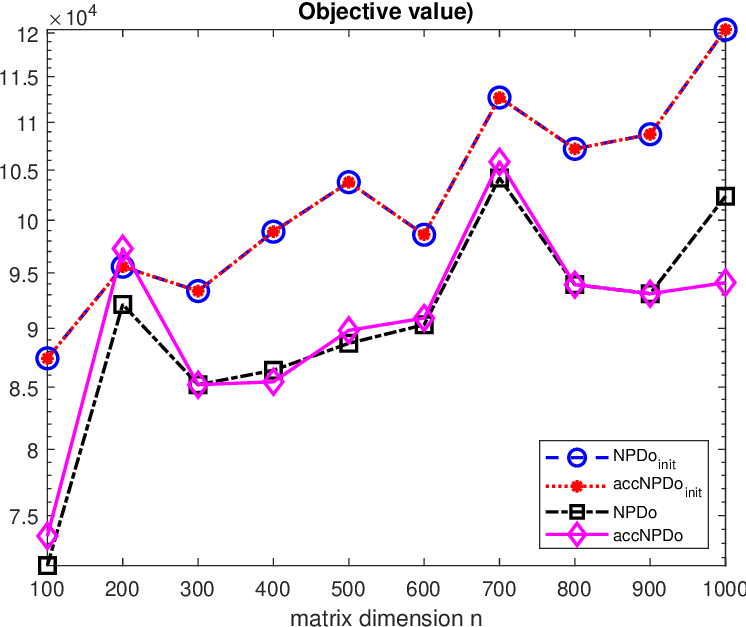}} \\ 
\rotatebox{90}{\hspace*{0.8cm}complex}
  & \resizebox*{0.22\textwidth}{0.14\textheight}{\includegraphics{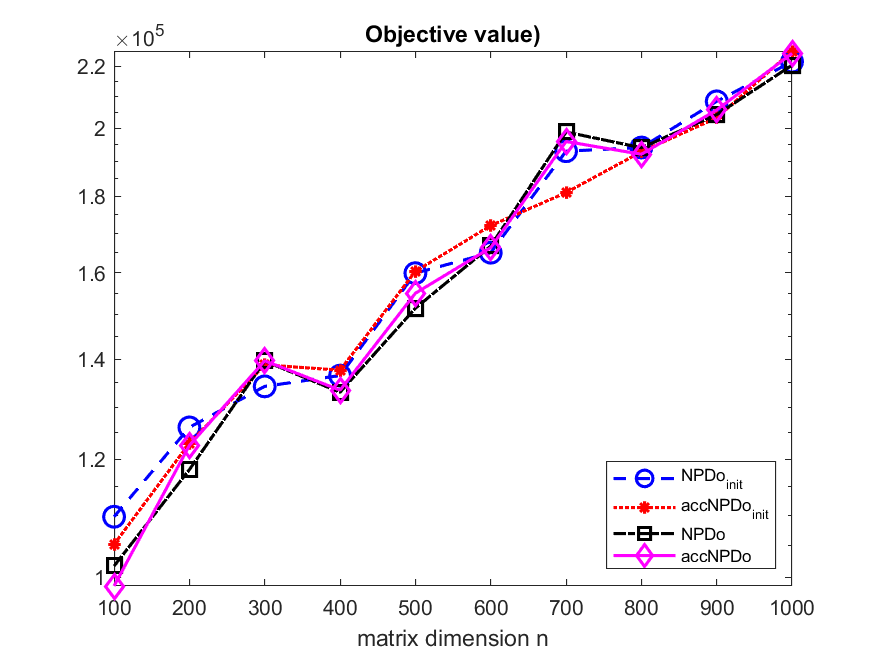}}
  & \resizebox*{0.22\textwidth}{0.14\textheight}{\includegraphics{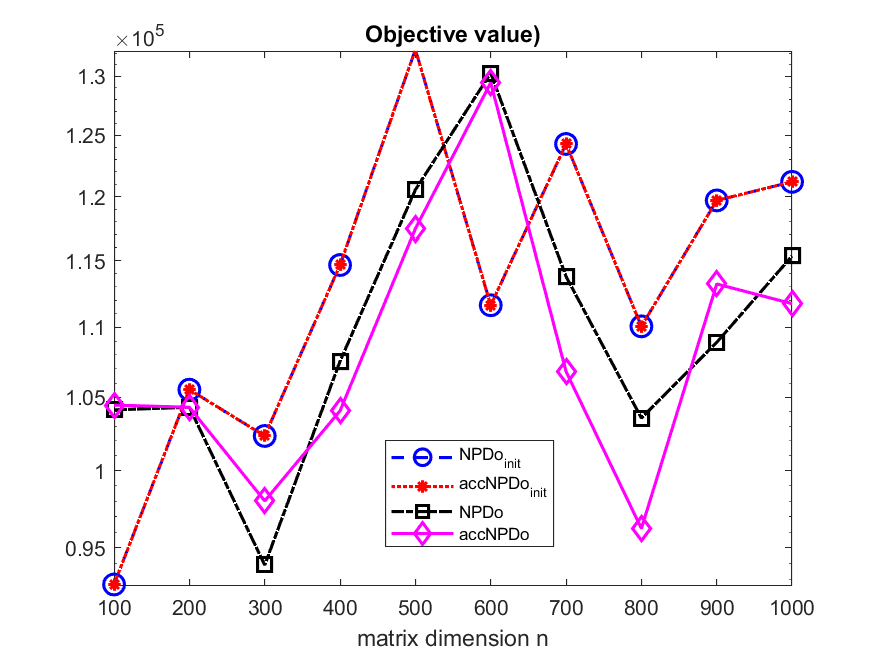}}
  & \resizebox*{0.22\textwidth}{0.14\textheight}{\includegraphics{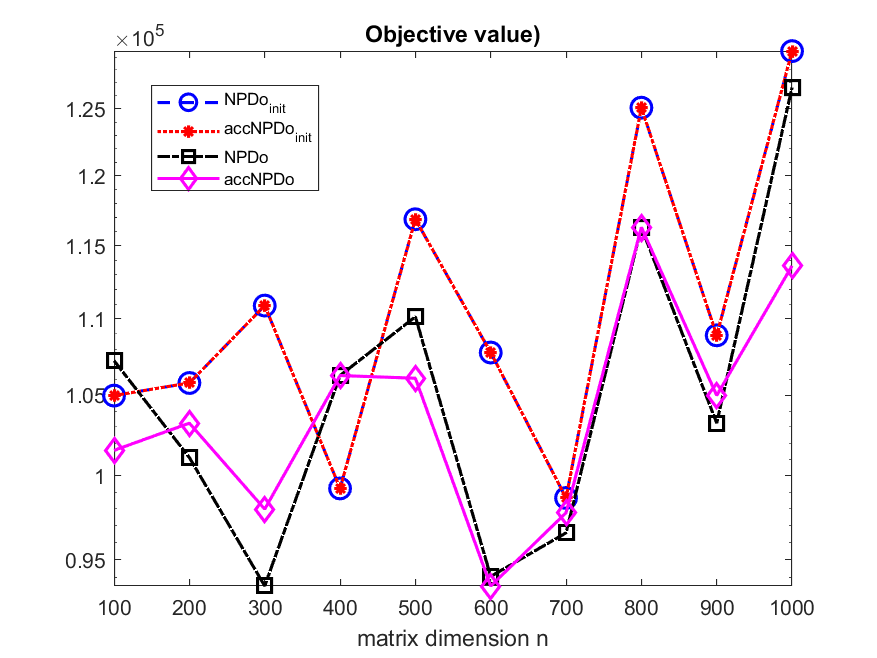}}
  & \resizebox*{0.22\textwidth}{0.14\textheight}{\includegraphics{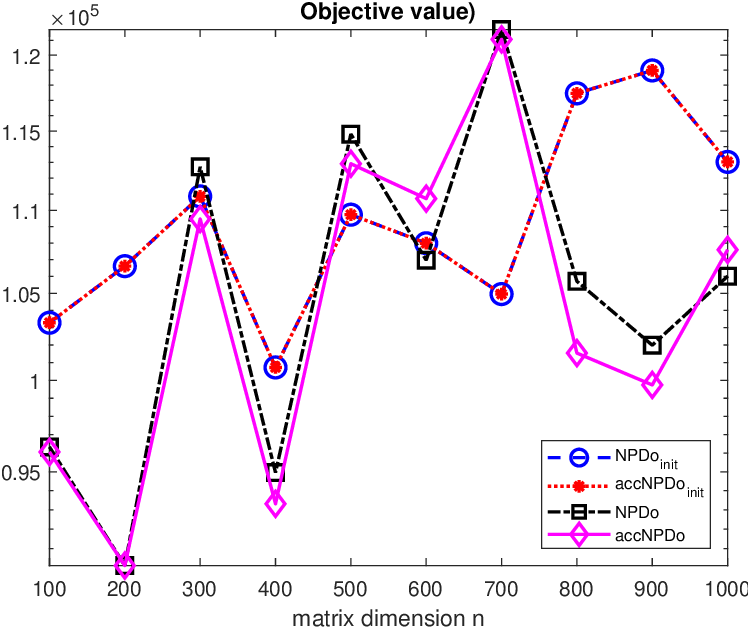}}
\end{tabular}\par
}
\vspace{-0.15 cm}
\caption{\small
   Effect of initial guess on the same testing sets that generates \Cref{fig:aJBD-n-vary}.
  }
\label{fig:aJBD-n-vary-init}
\end{figure}
Although there are a few times that random initials win, overwhelmingly, NPDo${}_{\init}$ and accNPDo${}_{\init}$ compute better stationary points
in terms of larger objective values.


\section{Conclusion}\label{sec:concl}
We studied the problem of computing the most dominant
joint block-diagonalizer for a given set $\{A_{\ell}\}_{\ell=1}^N$  of $N$ Hermitian matrices
in $\bbC^{n\times n}$,
that partially and unitarily block-diagonalizes all $A_{\ell}$ approximately and optimally.
To this end, we  propose to maximize the total joint
block-diagonal part under a common partial
congruent transformation by an othonormal matrix for all $A_{\ell}$:
\begin{equation}\label{eq:opt-pjbd'}
\max_{P\in\STM{k}{n}} \left\{ f(P):=\sum_{\ell=1}^N\|\BDiag_{\tau_k}(P^{\HH}A_{\ell}P)\|_{\F}^2\right\},
\end{equation}
where $1\le k\le n$ and $\tau_k=(k_1,\ldots,k_t)$ as in \eqref{eq:tau-n} specifies the joint block-diagonal structure.
In contrast to most existing approaches that focus on the case $k=n$
(i.e., the full joint block-diagonalization),
the new model  allows for the more general case of $k\le n$.
This is particularly important in modern applications involving
high dimensional data, where full block-diagonalization can be prohibitively
expensive and wasteful too, and a partial block-diagonalization with $k\ll n$ is
really what one needs in practice. We name it {\em principal joint block-diagonalization} (\pjbd) in general and
{\em principal block-diagonalization} (\pjd) which
aligns naturally with the goal of popular principal component analysis (PCA).

We propose to solve the {\pjbd} problem \eqref{eq:opt-pjbd'} by the NPDo approach
developed in~\cite{li:2024}, which is built on its strong
connection to numerical linear algebra to allow efficient computation.
This approach suits particularly well for large scale problems.
Two numerical solution techniques have been explored,
including the SCF iteration that globally converges
to a stationary point while produces a monotonically increasing convergent sequence
of the objective values,
and a subspace acceleration scheme based on the locally optimal conjugate
gradient technique.
Extensive numerical experiments demonstrated the
effectiveness of our methods and their superiority to Jacobi-type methods
for matrices of sizes $200$-by-$200$ or larger in the case of \pjd.

{\small
\def\noopsort#1{}\def\l{\char32l}\def\v#1{{\accent20 #1}}
  \let\^^_=\v\def\hbk{hardback}\def\pbk{paperback}

}

\end{document}